\documentclass[12pt]{article}
\usepackage{mathrsfs}
\usepackage{}

\usepackage{authblk}
\usepackage[noadjust]{cite}
\usepackage{amsmath}
\usepackage{amsthm,amsfonts,amssymb}           %math symbols and fonts
\usepackage{bbm}
\usepackage{enumerate}
\usepackage{wasysym}                           %small box \Box

\usepackage[margin=2.5cm
	       ]{geometry}
\usepackage[pdfstartview=FitH,
            CJKbookmarks=true,
            bookmarksnumbered=true,
            bookmarksopen=true,
            colorlinks,
            linkcolor=blue,
            anchorcolor=blue,
            citecolor=blue,
            urlcolor=blue
            ]{hyperref}

\DeclareFontFamily{OT1}{pzc}{}
\DeclareFontShape{OT1}{pzc}{m}{it}{<-> s * [1.10] pzcmi7t}{}
\DeclareMathAlphabet{\mathpzc}{OT1}{pzc}{m}{it}                 %mathcal lowercase, \mathpzc

\newtheorem*{lem2*}{Lemma 1.2*}
\newtheorem*{lem1*}{Lemma 1.1*}
\newtheorem*{thm1.3b}{Theorem 1.3b}
\newtheorem*{citedthm}{Theorem}

\newtheorem{thm}{Theorem}[section]

\newtheorem{lem}[thm]{Lemma}

\def\R{\mathbb R}

\def\valf{\mathcal Z}

\newcommand*{\dif}{\mathop{}\!\mathrm{d}}             %differtial operator
\newcommand{\sln}{\mathrm{SL}(n)}                  %SL(n)
\newcommand{\gln}{\mathrm{GL}(n)}                  %GL(n)

\def\MP{\mathscr{P}}
\def\MK{\mathscr{K}}
\def\MQ{\mathscr{Q}}

\newcommand\CF[1]{C(\R^{#1})}                                    %continuous function on $\R^n$
\newcommand\CFo[1]{C(\R^{#1}\setminus \{o\})}                  %continuous function on $\R^n \setminus \{o\}$
\newcommand{\ro}[1]{\R^{#1}\setminus \{o\}}                %\R^n \setminus \{o\}

\newcommand{\FF}[1]{F(\R^{#1}; \R)}                     %general functions
\newcommand{\FFo}[1]{F(\R^{#1}\setminus \{o\}; \R)}     %general functions on $\R^n \setminus \{o\}$
\newcommand{\cms}{\mathscr{M}^c(\R)}                   %signed and continuous Radon measures

\def\MKoon{\MK_{(o)}^n}
\def\MPon{\MP_{o}^n}

\newcommand\ab[1]{\left(#1\right)}        %auto bracket size

\def\e{\varepsilon}

\def\lap{\mathcal{L}}    %% \mathcal{L}\mathpzc{a} \, using defined \mathpzc command; see at beginning
\newcommand{\legt}[1]{{#1}^\ast}

\def\epi{\operatorname{epi}}
\def\epic{\overset{\operatorname{epi}}{\longrightarrow}}
\def\pc{\overset{\operatorname{p}}{\longrightarrow}}
\def\hc{\overset{\operatorname{hypo}}{\longrightarrow}}

\def\convs{\operatorname{Conv}_{sc}(\R^n)}
\def\convf{\operatorname{Conv}(\R^n; \R)}

\def\cvx{\operatorname{Conv}(\R^n)}
\def\lc{\operatorname{LC}(\R^n)}
\def\lcsc{\operatorname{LC}_{sc}(\R^n)}
\def\lcpos{\operatorname{LC}(\R^n;(0,\infty))}
\def\dom{\operatorname{dom}}

\def\u{u}                                   %convex function
\def\v{v}                                   %convex function
\newcommand{\lf}{\ell}                      %linear function
\def\indf{\mathbf{I}}
\def\dom{\operatorname{dom}}

\def\sgn{\operatorname{sgn}}

\newcommand\expfsmall[1]{e^{#1}}
\newcommand\expf[1]{\exp\left\{#1\right\}}

\def\simga{\sigma}
\def\myvskip{\vskip 5pt}

\makeatletter
\newcommand{\Rmnum}[1]{\expandafter\@slowromancap\romannumeral #1@} %大写罗马数字
\makeatother

\makeatletter
\newcommand{\subjclass}[2][1991]{%
  \let\@oldtitle\@title%
  \gdef\@title{\@oldtitle\footnotetext{#1 \emph{Mathematics subject classification.} #2}}%
}
\newcommand{\keywords}[1]{%
  \let\@@oldtitle\@title%
  \gdef\@title{\@@oldtitle\footnotetext{\emph{Key words and phrases.} #1.}}%
}
\makeatother

\title{\bf{The Legendre transform, the Laplace transform and valuations}}
\author[1,2,3]{Jin Li}
\affil[1]{Department of Mathematics, Shanghai University, Shanghai, China, 200444}
\affil[2]{Newtouch Center for Mathematics of Shanghai University, Shanghai, China,  200444}
\affil[3]{Institut f\"{u}r Diskrete Mathematik und Geometrie, Technische Universit\"{a}t Wien, Wien, Austria, 1040 \authorcr \href{mailto: Jin Li<li.jin.math@outlook.com>}{li.jin.math@outlook.com}}
\date{}

\begin{document}

\maketitle

\begin{abstract}
We first prove that the Legendre transform is the only continuous and $\mathrm{SL}(n)$ contravariant valuation that behaves as a conjugation of two important translations on super-coercive, lower semi-continuous, and convex functions.
Then we turn to a similar setting on log-concave functions and find characterizations of not merely the duality transform but also the Laplace transform on log-concave functions.
With the notion of dual valuation, we also obtain characterizations of the identity transform on finite convex functions and positive log-concave functions.
\end{abstract}

\section{Introduction}
Ever since Dehn's solution to Hilbert's third problem and Hadwiger's characterization of (intrinsic) volumes, valuation theory has been widely and deeply studied.
Various important operators on convex bodies were characterized in valuation theory with their natural geometric invariances; see for example \cite{Ale99,Ale01,Ale04,AF2013con,Lud05,Sch08,klain2000even,LR10,BF2011herm,HP14b,SW2015mink,McM77,ABS2011harm,Sch2010,HS2014loc,Lud06,Li2020slnco,Li2018AFV,Kla95,Lud10b,MR4295088,MR4316669}.
See also the books and surveys \cite{KR97,Lud2006,MR3380549,Schb2}.
The theory of valuations on function spaces is a novel, developing field \cite{CLM2017Min,Tsa10,CLM2017RV,Lud12,Lud11b,CLM2019homogeneous,CLM2020hadwiger,CLMhadwiger2,CLMhadwiger3,CLMhadwiger4,MR4070303,MR3783417,MR4252807,Ma15real,MR3897436,MR4201535}, also see the surveys \cite{Lud11a,ludwig2021geometric}.
However, it is somehow surprising that only a few studies on transforms of functions in valuation theory were made (see \cite{LM2017Lap,MR3870609,MR4567496} for some recent results) although many important transforms are valuations.
The main aim of this paper is to establish certain characterizations of the Legendre transform and the Laplace transform with similar assumptions.

For background and applications of the Legendre transform and the Laplace transform, see, e.g., \cite{Roc70conv,Doe74}.
The Laplace transform on Lebesgue functions was characterized previously in \cite{LM2017Lap}.
However, it is a different challenge for log-concave functions.

We remark that a series of beautiful characterizations of transforms of functions, including the Legendre transform and the Fourier transform (closely connected to the Laplace transform), was established by Artstein-Avidan, Milman, and many others, for example, \cite{MR2928347,AM09Leg,MR2719283,MR2985126,AM11hid,MR4372149}.
Although the assumptions in their papers are related to valuations, their results are somehow far away from this study.
We will explain it later for the Legendre transform.

Denote by $\R^n$ the $n$-dimensional Euclidean space.
Let $\Gamma$ be a subset of a lattice $(\Gamma',\vee,\wedge)$.
We call a map $\valf$ mapping from $\Gamma$ to an Abelian semigroup $(\mathbb{A},+ )$ a \emph{valuation} if
\begin{align*}
\valf (\gamma_1 \vee \gamma_2) + \valf (\gamma_1 \wedge \gamma_2) = \valf \gamma_1 + \valf \gamma_2,
\end{align*}
whenever all four elements $\gamma_1, \gamma_2, \gamma_1 \vee \gamma_2, \gamma_1 \wedge \gamma_2 \in \Gamma$.
If $\Gamma$ is a space of (extended) real-valued functions on $\R^n$, then
\begin{align*}
(\gamma_1 \vee \gamma_2)(x) = \max\{\gamma_1(x),\gamma_2(x)\},~~
(\gamma_1 \wedge \gamma_2)(x) = \min\{\gamma_1(x),\gamma_2(x)\}
\end{align*}
for any $x \in \R^n$.
If $\Gamma$ is a set of convex bodies, then $\vee$ and $\wedge$ are the union and intersection of convex bodies, respectively.
We remark that in some previous characterization of dualities (e.g. \cite{MR2438994,AM09Leg}), different lattice structures are often used, in which the operation $\wedge$ is replaced by $\tilde \wedge$ for convex functions (and similarly for convex bodies).
For convex functions $\u$ and $\v$, the function $(\u \tilde \wedge \v)(x)$ is defined as the greatest convex function that is smaller than $\u \wedge \v$.
Clearly $\u \wedge \v = \u \tilde \wedge \v$ if and only if $\u \wedge \v$ is convex.
Our definition of valuation aligns with that in previous works on valuations, which do not impose additional conditions when $\u \wedge \v$ is not convex. This allows the framework to include a wide class of important maps (e.g. the Laplace transform introduced later and Theorem \ref{thm:log} for convex functions correspondingly) as valuations.

Let $\cvx$ be the space of all lower semi-continuous, convex functions $\u:\R^n \to \R \cup \{\infty\}$.
We consider two of its important subspaces.
Let $\convs$ be the space of all proper, super-coercive $\u\in \cvx$ and let $\convf$ be the space of all finite convex functions $\u: \R^n \to \R$.
Here a convex function $\u$ is called \emph{proper} if $\u \not\equiv \pm \infty$; and it is called \emph{super-coercive} if $\lim_{|x|\to \infty}\frac{\u(x)}{|x|}=\infty$.
The functions within $\convs$ carry important significance in convex geometric analysis as counterparts to compact convex sets, since Legendre transforms of these functions result in finite convex functions.
Furthermore, the Legendre transform also maps finite convex functions back to super-coercive convex functions.
Consequently, we can utilize what are known as dual valuations to seamlessly translate the classification of valuations on $\convs$ into classifications of valuations on $\convf$.
For more details, see \S \ref{sec:dualv}.

The classical \emph{Legendre transform} of $\u \in \cvx$ is
\begin{align*}
\legt{\u} (x)= \sup_{y \in \R^n} \langle x,y\rangle - \u (y),~ x\in \R^n.
\end{align*}
Artstein-Avidan and Milman \cite{AM09Leg} showed that the Legendre transform is essentially the only transform $\valf :\cvx \to \cvx$ that is a bijection and an order conjugation (they called this an order-reversing isomorphism), that is,
\begin{align*}
\u \leq \v \Longleftrightarrow \valf \u \geq \valf \v.
\end{align*}
Since $\valf :\cvx \to \cvx$ is a bijection, being an order conjugation is equivalent to the following conjugation of ``minima" and maxima:
\begin{align*}
\valf (\u \vee \v) = (\valf \u) \tilde \wedge (\valf \v),~~\valf (\u \tilde \wedge  \v) = (\valf \u)  \vee (\valf \v).
\end{align*}
(the equivalence is also pointed out in \cite{AM09Leg} and is a key step in their proof.) 
If $\u \wedge \v$ is convex, then $\legt{\u} \wedge \legt{\v}= \legt{\u} \tilde \wedge \legt{\v}$; see \cite{CLM2017hessian}[Lemma 3.4].
Therefore, for the Legendre transform, being a conjugation of ``minima" and maxima is much stronger than being a valuation.

We require other natural properties to characterize the Legendre transform in valuation theory, so it is helpful first to recall the characterization of its analog in the theory of convex bodies: the polar body map on convex bodies (that associates with a convex body $K$ its polar body $K^\ast$).
Such characterizations in valuation theory were established by Ludwig \cite{Lud06,Lud10b}.
We would also like to remark that a characterization of the polar body map analogous to \cite{AM09Leg} was established by B\"{o}r\"{o}czky and Schneider \cite{MR2438994}; see also, Gruber \cite{MR1182848}, Artstein-Avidan and V.~Milman \cite{MR2406688} and Slomka \cite{MR2795422}.
Let $\MK^n$ be the set of all convex bodies in $\R^n$, and $\MKoon$ be the set of all convex bodies containing the origin in their interiors.

\begin{citedthm}[Ludwig \cite{Lud06,Lud10b}]
Let $n \ge 2$.
A map $Z:\MKoon \to (\MK^n,+)$ is a continuous valuation satisfying
\begin{align*}
Z(\phi K)=\phi^{-t}ZK
\end{align*}
for every $K \in \MKoon$ and $\phi \in \gln$,
if and only if there are $c_1,c_2 \ge 0$ such that
\begin{align*}
ZK=c_1 K^\ast + c_2 (-K^\ast)
\end{align*}
for every $K\in \MKoon$.
Here ``$+$" can be either Minkowski addition or radial addition.
\end{citedthm}

However, the following example shows that we can not characterize the Legendre transform as analogous to Ludwig's theorem without additional assumptions.
Let $\FF{n}$ be the set of all functions $f:\R^n \to \R$
and let $\Gamma$ be a space of (extended) real-valued functions on $\R^n$ which is invariant under transformations in $\sln$, that is,
if $\gamma \in \Gamma$, then $\gamma \circ \phi^{-1} \in \Gamma$ for all $\gamma \in \sln$.
We say that a transform $\valf :\Gamma \rightarrow \FF{n}$ is \emph{$\sln$ contravariant} (or \emph{$\gln$ contravariant}) if
\begin{align*}
\valf (\gamma \circ \phi^{-1})= (\valf \gamma) \circ \phi^{t}
\end{align*}
for every $\gamma \in \Gamma$ and $\phi \in \sln$ (or $\phi \in \gln$ respectively).
Here $\sln$ is the special linear group and $\gln$ is the group of linear transformations of $\R^n$.

\myvskip
\noindent\textbf{Example.}
Define $\widetilde \valf: \convs \to \FF{n}$ by
\begin{align*}
\widetilde \valf u(x)=\int_{0}^\infty h(\{\expfsmall{-u} \ge t\},x) \dif t, ~ u \in \convs,~x \in \R^n.
\end{align*}
Here $h(\{\expfsmall{-u} \ge t\},\cdot)$ is the support function of $\{\expfsmall{-u} \ge t\}:=\{y \in \R^n: \expfsmall{-u(y)} \ge t\}$; see the definition of the support function in \S \ref{s2}.
It is easy to see that both $\widetilde \valf$ and the Legendre transform are $\gln$ contravariant valuations.

To avoid this example, we consider what happens if we translate functions.
There are two fundamental translations of convex functions.
One is the (usual) translation $\tau_y \u(x):=\u(x-y)$ for $y \in \R^n$ and $\u \in \cvx$.
The second one is the so-called dual translation $u+\lf_{y}$ for $y \in \R^n$ and $\u \in \cvx$; see, e.g., \cite{Ale2017MAo,CLM2019homogeneous,MR4070303,CLM2020hadwiger} (sometimes with different name).
Here $\lf_{y}(x)=x \cdot y$ for any $x \in \R^n$, where $x \cdot y$ is the inner product of $x,y \in \R^n$.
The Legendre transform behaves as a conjugation of these two translations; that is, it maps the translation of a function to the dual translation of the valued function and vice versa.
However, the above example $\widetilde \valf$ only maps the translation to the dual translation (with some weight as the assumption in Theorem \ref{thm:lt2a}), lacking the reverse direction.
In fact,
\begin{align*}
\widetilde \valf (\tau_y \u)(x)
&=\int_0^{\max_{x\in \R^n}\expfsmall{-\u(x)}} h(\{\expfsmall{-\tau_y\u} \ge t\},x) \dif t\\
&=\int_0^{\max_{x\in \R^n}\expfsmall{-\u(x)}} h(\{\expfsmall{-\u} \ge t\},x)+x\cdot y \dif t\\
&=\widetilde \valf (\u)(x) +  e^{-\min_{x\in\R^n} \u(x)} \lf_y(x).
\end{align*}

We say that a transform $\valf :\convs \rightarrow \FF{n}$ is a \emph{translation conjugation} if
\begin{align*}
\valf (\tau_y \u)=\valf (\u)+\lf_{y}, ~\valf (\u+\lf_{y})=\tau_{y} \valf (\u)
\end{align*}
for every $u \in \convs$ and $y \in \R^n$.
It is called \emph{continuous} if $\u_i \epic \u$ implies $\valf \u_i \pc \valf \u$.
Here $\u_i \epic \u$ denotes that $\u_i$ epi-converges to $\u$ and $\valf \u_i \pc \valf \u$ denotes that $\u_i$ converges to $\u$ pointwise; see \S \ref{s2}.
In the following, we mostly work in the space $\R^n$ for $n \ge 3$ since we use a result in \cite{Li2020slnco}; see the following Lemma \ref{lem:3body}.
The planar cases not included in this paper will be studied separately by different approaches.
We obtain the following characterization of the Legendre transform.

\begin{thm}\label{thm:Leg}
Let $n \ge 3$. A transform $\valf :\convs \rightarrow \FF{n}$ is a continuous and $\sln$ contravariant valuation which is a translation conjugation, if and only if there is a constant $c \in \R$ such that
\[\valf \u = \legt{\u} + c\]
for every $\u \in \convs$.
\end{thm}

We remark that in Artstein-Avidan and Milman \cite{AM09Leg} it is assumed that the transform $\valf$ is bijective.
Here, we assume neither injectivity nor surjectivity.

We will see in the following Theorem \ref{thm:lt2} and Theorem \ref{thm:lt2a} that the result does not change too much if we slightly modify our assumption of translation conjugation.
However, if we apply similar ideas to log-concave functions, it will be somehow surprisingly different since the Laplace transforms also appears.
In fact, the Laplace transform provides a continuous and $\sln$ contravariant valuation on $\convs$ that only satisfies the second relation of translation conjugation, but does not satisfy the first one; see the following Theorem \ref{thm:log} for the case $\e=\sigma = 1$.

Let $\lc:=\{e^{-\u}: \u \in \cvx\}$, that is, the space of all upper semi-continuous, log-concave functions $f:\R^n \to [0,\infty)$.
Let $\lcsc:=\{e^{-\u}: \u \in \convs\}$.
Many concepts and results on log-concave functions and on convex functions are in one-to-one correspondence.
For instance, the duality of the log-concave function $f=e^{-\u} \in \lc$ is
$$f^\circ:=e^{-\u^\ast},$$
and the usual translation $\tau_y f =e^{-\tau_y \u}$; see \cite{AKM04}.
We say that $e^{-\lf_y}f$ is the \emph{dual translation} of $f=e^{-u}$ since $\u + \lf_y$ is the dual translation of $\u \in \convs$.
We call a transform $\valf :\lcsc \rightarrow \FF{n}$ a \emph{translation conjugation on log-concave functions} if
\begin{align}\label{eq:log}
\valf (\tau_y f)= \expfsmall{-\lf_{y}} \valf (f),~\valf (\expfsmall{-\lf_{y}}f)=\tau_y \valf (f)
\end{align}
for every $f \in \lcsc$ and $y \in \R^n$;
and it is \emph{continuous} if $f_i \hc f$ implies $\valf f_i \pc \valf f$.
Here $f_i \hc f$ denotes that $f_i$ hypo-converges to $f_i$, which is equivalent to $(-\log f_i) \epic (-\log f)$.
We also obtain the following characterization of $f^\circ$.

\begin{thm}\label{mthm:log1}
Let $n \ge 3$. A transform $\valf :\lcsc \rightarrow \FF{n}$ is a continuous and $\sln$ contravariant valuation which is a translation conjugation on log-concave functions, if and only if there is a constant $c \in \R$ such that
\begin{align*}
\valf f= c f^\circ
\end{align*}
for every $f \in \lcsc$.
\end{thm}

If we slightly modify the assumption \eqref{eq:log}, then the Laplace transform appears.
Here the Laplace transform of $f \in \lcsc$ is
\begin{align*}
\lap f (x):=\int_{\R^n} \expfsmall{x \cdot y} f(y)\dif y, ~x\in \R^n.
\end{align*}
This version of the Laplace transform was introduced by Klartag and Milman \cite{KM12} which differs from the standard version $\int_{\R^n} \expfsmall{-x \cdot y} f(y)\dif y$
(but there is no essential difference in the characterizations).

\begin{thm}\label{mthm:log}
Let $n \ge 3$. A transform $\valf :\lcsc \rightarrow \FF{n}$ is a continuous and $\sln$ contravariant valuation which satisfies
\begin{align*}
\valf (\tau_y f)= \expfsmall{\lf_{y}} \valf (f),~\valf (\expfsmall{-\lf_{y}}f)=\tau_y \valf (f)
\end{align*}
for every $f \in \lcsc$ and $y \in \R^n$,
if and only if there are constants $c_1,c_2 \in \R$ such that
\begin{align*}
\valf f= \frac{c_1}{f^\circ} + c_2 \lap f
\end{align*}
for every $f \in \lcsc$.
\end{thm}

Theorem \ref{thm:Leg} and Theorem \ref{mthm:log1} look similar but are essentially quite different.
To have a better understanding of the difference between these two theorems, we can transfer Theorem \ref{mthm:log1} to convex functions as the following Theorem \ref{thm:log} for the case $\e=1$ and $\sigma = -1$ (see short explanations before Theorem \ref{thm:log}).

The rest of the paper is organized as the following.
Some notation, background information, and results related to the Cauchy functional equation are collected in the second and third sections.
Then in \S \ref{sec:vb}, we establish some classifications of valuations on convex bodies connected to the main results.
General theorems including the main results are proved in \S \ref{sec:mf}.
Finally, we utilize dual valuations to find characterizations of the identity transform.

\section{Preliminaries and Notation}\label{s2}
Let $x, y \in \R^n$. The inner product of $x$ and $y$ is denoted by $x \cdot y$.
For $x \in \ro{n}$ and $t \in \R$, the hyperplane $H_{x,t}:=\{y \in \R^n: x \cdot y =t\}$.
The two half-spaces bounded by $H_{x,t}$ are $H_{x,t}^+:=\{y \in \R^n: x \cdot y \ge t\}$ and $H_{x,t}^-:=\{y \in \R^n: x \cdot y \le t\}$.

For $A_1,A_2 \subset \R^n$, we denote the convex hull of $A_1 \cup A_2$ by $[A_1,A_2]$.
Sometimes we do not distinguish between a point and a singleton.

Let $\{e_1,\dots,e_n\}$ be the standard Euclidean basis of $\R^n$.
For $r \le s$, we denote $[r,s]=[re_1,se_1]$ and $[r,s]^n=\sum_{i=1}^n [re_i,se_i]$.
Also, $[r,s] \times [r',s']^{n-1}=[re_1,se_1]+\sum_{i=2}^n [r'e_i,s'e_i]$ for any $r \le s$ and $r' \le s'$.

For $r \in \R$,
\begin{align*}
  \sgn(r)
  :=\begin{cases}
    1, & r>0, \\
    -1, & r<0, \\
    0, &r=0.
  \end{cases}
\end{align*}

The domain of a convex function $\u:\R^n \to (-\infty,\infty]$ is $\dom \u:=\{x \in \R^n: \u(x) < \infty\}$.
The \emph{epigraph} of $\u$ is $\epi \u :=\{(x,t)\in\R^{n+1}:\u(x) \leq t, x \in \dom \u\}$.
The epigraph of $\u \vee \v$ is the intersection of $\epi \u \cap \epi \v$; the epigraph of $\u \wedge \v$ is $\epi \u \cup \epi \v$, and the epigraph of $\u \tilde \wedge \v$ is $[\epi \u, \epi \v]$.

For references to epi-convergence, see \cite{Roc70conv,RW1998VA}.
Here we only recall the definition and some properties needed in this paper.
Let $\{u_i\}_{i=1}^\infty \cup \{\u\} \subset \cvx$.
We say that the sequence of $\u_i$ \emph{epi-converges} to $\u$ ($\u_i \epic \u$) if $\epi u_i \to \epi u$ in the sense of Kuratowski
convergence, that is, $\forall x \in \epi \u$, there exists $x_i \in \epi \u_i$ such that $x_i \to x$; and if $x_{i_j} \in \epi \u_{i_j}$ such that $x_{i_j} \to x$, then $x \in \epi \u$.
Equivalently, $u(x) \leq \liminf _{i \rightarrow \infty} u_{i}\left(x_{i}\right)$ for all $x_i \to x$;
and there exists $x_i \to x$ such that $u(x)=\lim _{i \rightarrow \infty} u_{i}\left(x_{i}\right)$.
We remark that $\epi u_i \to \epi u$ if and only if $\u_i \to \u$ uniformly on any compact set contained in the interior of $\dom \u$.

The Legendre transform is epi-continuous on $\cvx$, that is, $\u_i \epic \u$ implies that $\legt{(\u_i)} \epic \legt{\u}$ for $\u_i, \u \in \cvx$.
For $\convf$, the topology induced by epi-convergence is equivalent to that induced by pointwise convergence.
Since the Legendre transform maps $\convs$ to $\convf$, it shows that $\u_i \epic \u$ implies that $\legt{(\u_i)} \pc \legt{\u}$ for $\u_i, \u \in \convs$.

For the well-definedness and continuity of the Laplace transform on $\lcsc$, we need the following lemma, which seems folklore. But we do not find a good reference and hence give a proof here.
\begin{lem}
Let $\u \in \convs$. Then for all $a>0$, there exists $b \in \R$ such that $\u(y) > a|y|+b$ for all $y \in \R^n$.
Moreover, assume that $\u_i \epic \u$ and $a,b$ satisfying the above condition for $\u$ are fixed. Then $\u_i(y) > a|y|+b$ for all $y \in \R^n$ when $i$ is sufficiently large.
\end{lem}
\begin{proof}
Since $\u$ is super-coercive, the $\min_{x \in \R^n} \u(x)$ exists and is finite. 
We may translate the epigraph of $\u$ so that $\u(o)=\min_{x \in \R^n} \u(x)=0$ and only prove the lemma for this case.

Assume that the first part of the lemma is not true.
Then there exists $a>0$ such that for all $b \in \R^n$, there exists $y_b \in \R^n$ satisfying $\u(y_b) \le a|y_b| + b$.
Then starting from an arbitrary $b_1 \in \R^n$, we find $y_1 \in \R^n$ such that $\u(y_1) \le a|y_1| + b_1$.
The level set $\{x \in \R^n : \u(x) \le \u(y_1)\}$ is a convex compact set since $\u \in \convs$.
Denote by $A_1$ the convex hull of this level set and the ball of $\R^n$ with center $o$ and radius $1$.
Then we can find $b_2 < b_1$ such that $\u(x) \ge 0 > a|x| + b_2$ for all $x \in A_1$.
But by our assumption, there is $y_2 \in \R^n$ such that $\u(y_2) \le a|y_2| + b_2$.
Hence $y_2 \notin A_1$.
Continuing this process, we obtain a sequence $\{y_k\}_{k=1}^\infty$ in $\R^n$ such that 
$$\u(y_k) \le a|y_k| + b_k$$ 
and $b_k < b_{k-1}$ for all $k \ge 2$, and $y_{k} \notin A_{k-1}$, where $A_{k-1}$ is the convex hull of the level set $\{x \in \R^n : \u(x) \le \u(y_{k-1})\}$ and the ball of $\R^n$ with center $o$ and radius $k-1$.
Therefore, $|y_k| \to \infty$ as $k \to \infty$.
However, $\frac{\u(y_k)}{|y_k|} \le a + \frac{b_k}{|y_k|} \le a + b_2$ for $k \ge 2$.
It contradicts $\lim_{|y| \to \infty} \frac{\u(y)}{|y|} = \infty.$

For the second part of the lemma, denote $v(x):=a|x|+b$ for $x \in \R^n$. 
We note that $\epi \u$ is contained in the interior of $\epi v$ by the first part. 
Since $\u_i \epic \u$, we have $\epi \u_i \to \epi \u$ in the sense of Kuratowski and hence $\epi \u_i$ is also contained in the interior of $\epi v$ when $i$ is sufficiently large.
Thus $\u_i(y) > a|y|+b$ for all $y \in \R^n$ when $i$ is sufficiently large.
\end{proof}

By the previous Lemma, for any given $u \in \convs$ and $x \in \R^n$, we can find $a>0$ and $b \in \R$ such that $\u(y) > (|x|+a)|y|+b$ for all $y \in \R^n$.
Thus
\begin{align*}
\lap \ab{\expfsmall{-\u}} (x)
=\int_{\R^n} \expfsmall{x \cdot y -  \u(y)}\dif y
&\le \int_{\R^n} \expfsmall{|x| |y| -  (|x|+a)|y|-b}\dif y \\
&=\int_{\R^n} \expfsmall{-a|y|-b}\dif y \\
&=\int_{0}^{\infty} \expfsmall{-ar-b} r^{n-1} n \omega_n \dif r \\
&=n! a^{-n}\omega_n \expfsmall{-b}<\infty,
\end{align*}
where $\omega_n$ is the volume of $n$-dimensional unit ball.
Let $\u_i, \u \in \convs$ such that $\u_i \epic \u$.
We also have $\u_i(y) > (|x|+a)|y|+b$ for all $y \in \R^n$ when $i$ is sufficiently large.
The epi-convergence of $\u_i$ implies $\u_i \pc \u$ a.e.
Thus, $\u_i \epic \u$ implies that $\lap \expfsmall{-\u_i} \pc \lap \expfsmall{-\u}$ by the dominated convergence Theorem.

The support function $h_K$ and the indicator function $\indf_K$ of $K \in \MK^n$ are two fundamental classes of convex functions.
Their definitions are the following: for $x \in \R^n$,
\begin{align*}
h_K(x):=\max_{y \in K} x \cdot y,
\end{align*}
and
\begin{align*}
\indf_K(x):=\begin{cases}
0, & x\in K, \\
\infty, &x \notin K.
\end{cases}
\end{align*}

We list here some basic properties of the Legendre transform (for all $\u, \v \in \cvx$).
Some have been discussed above, and we omit the proofs of the others.
\begin{enumerate}[(D1)]
  \item \label{pval} Valuation.

  \item $\legt {(\u \circ \phi^{-1})}= \legt{\u} \circ \phi^{t}$, for all $\phi \in \sln$.

  \item $\legt{(\tau_y \u)}=\legt{\u}+\lf_{y}$, for all $y \in \R^n$.

  \item $\legt{(\u+\lf_{y})}=\tau_{y} \legt{\u}$, for all $y \in \R^n$.

  \item $\legt {(\u + t)} =\legt{\u} - t$, for all $t \in \R$.

  \item $\legt {(\u \circ \lambda)} = \legt{\u} \circ \frac{1}{\lambda} $, for all $\lambda \neq 0$.

  \item $\legt {(\lambda \u)} = \lambda (\legt{\u} \circ \frac{1}{\lambda})$, for all $\lambda \neq 0$.

  \item Continuity: $\u_i \epic \u$ implies $\legt{(\u_i)} \epic \legt{\u}$.

  \item \label{pbi} Bijection: $\cvx \to \cvx $.

  \item $\legt{(\legt{\u})} = \u$.

  \item \label{pmon} $\u \leq \v \Longleftrightarrow \legt{\u} \geq \legt{\v}$.

  \item \label{pmax} $\legt{(\u \vee \v)} = \legt{\u} \tilde \wedge \legt{\v}$, $\legt{(\u \tilde \wedge  \v)} = \legt{\u} \vee \legt{\v}$.

  \item \label{phom} $\legt{(\u \Box \v)} = \legt{u} + \legt{\v}$.

  \item \label{pid} $\legt \indf_K = h_K$, for all $K\in \MK^n$.
\end{enumerate}

Here $(\u \circ \lambda)(x) := \u(\lambda x)$ for any $\lambda \in \R$ and $x \in \R^n$, and 
\begin{align*}
\u \Box \v (x) := \inf_{x_1+x_2=x} \u(x_1)+\v(x_2), ~x \in \R^n,
\end{align*}
or equivalently,
\begin{align*}
\epi (\u \Box \v) = \epi \u + \epi \v.
\end{align*}

The characterization of the Legendre transform was established by Artstein-Avidan and Milman \cite{AM09Leg} involving (D\ref{pbi})-(D\ref{pmax}).
Another characterization of the Legendre transform was established by Rotem \cite{MR3062743} using (D\ref{pmon}), (D\ref{phom}) and (D\ref{pid}) (in fact, a weaker assumption than (D\ref{phom}) was used and ``$\Box$" was also characterized).
It is easy to see that (D\ref{phom}) implies (D\ref{pval}).

Sometimes we write the exponential function as $\expf{\cdot}$.
We also list some basic properties of the Laplace transform on $\lcsc$ (for all $f \in \lcsc$).
\begin{enumerate}[(D1)]
  \item \label{plapval} Valuation.

  \item $\lap (f \circ \phi^{-1})= (\lap f) \circ \phi^{t}$, for all $\phi \in \sln$.

  \item \label{plaplogtran} $\lap (\tau_y f)=\expf{\lf_{y}}\lap f$, for all $y \in \R^n$.

  \item $\lap (\expfsmall{-\lf_{y}}f)=\tau_{y} \lap f$, for all $y \in \R^n$.

  \item
  $\lap (\expfsmall{-t} f) = \expfsmall{- t} \lap f$, for all $t \in \R$.

  \item \label{plaphom1} $\lap (f \circ \lambda) = \lambda^{-n} (\lap f) \circ \frac{1}{\lambda} $, for all $\lambda \neq 0$.

  \item \label{plaphom2} $\lap (\lambda f) = \lambda(\lap f) $, for all $\lambda \ge 0$.

  \item \label{plapcont} Continuity: $f_i \hc f$ implies $\lap f_i \pc \lap f$.
\end{enumerate}

For compactly supported $L^1$ functions on $\R^n$, the Laplace transform was characterized by the author jointly with Ma \cite{LM2017Lap} using properties (D\ref{plapval})-(D\ref{plaplogtran}), (D\ref{plaphom1}), and (D\ref{plaphom2}) where the continuity hypothesis is according to the $L^1$-norm.

\section{The Cauchy functional equation and related results}
We say that a function $f:[0,\infty) \to \R$ satisfies the \emph{Cauchy functional equation} if
\begin{align*}
f(a+b) = f(a) + f(b),
\end{align*}
for every $a,b \ge 0$.
If we further assume that $f$ is continuous, then there is a constant $c\in \R$ such that
\begin{align*}
f(r) = cr
\end{align*}
for every $r \ge 0$.
Therefore, we easily get the following.
\begin{lem}\label{lem:220515-1}
If a continuous function $\mathcal{A}:\R^n \to \R^n$ satisfies $\mathcal{A}(x+y)=\mathcal{A}(x)+\mathcal{A}(y)$
for every $x,y \in \R^n$, then $\mathcal{A}$ is a linear transformation on $\R^n$.
\end{lem}

We also get the following useful lemma.
\begin{lem}\label{lem:220515-2}
Let $\sigma \in \R$.
If $f_1,f_2:[0,\infty) \to \R$ are two continuous functions satisfying
\begin{align}\label{eq220515-1}
f_1(s)+f_2(r) - \sigma t = f_1(s+ t) + f_2(r-t)
\end{align}
for every real $r \ge 0 \ge t \ge -s$, then there are constants $c_1,c_2 \in \R$ such that
\begin{align}
&f_1(s)=c_1 s + f_1(0),~\forall s\ge0, \label{eq220520-1}\\
&f_2(r)=c_2 r + f_2(0),~\forall r\ge0, \label{eq220520-2}
\end{align}
and $c_1 - c_2 +\sigma=0$.
\end{lem}
\begin{proof}
Reformulating \eqref{eq220515-1}, we get
\begin{align*}
f_1(s)-f_1(s+t) - \sigma t = -f_2(r) +f_2(r- t).
\end{align*}
Since the right-hand side of the above relation does not rely on $s$, and the left-hand side of the above relation does not rely on $r$, there is a function $g:[0,\infty) \to \R$ such that
\begin{align}
&g(t) = f_1(s)-f_1(s+ t) - \sigma t \label{eq220515-2}
\end{align}
for every $s \ge -t \ge 0$.

Set $s=-t$ in \eqref{eq220515-2}.
We have $g(t)=f_1(- t) - f_1(0) - \sigma t$.
Replacing the left-hand side of \eqref{eq220515-2} by $f_1(- t) - f_1(0) - \sigma t$, we further get
\begin{align*}
f_1(s) = f_1(s+ t) + f_1(- t) - f_1(0).
\end{align*}
Denote $a:=s+ t$ and $b:= - t$.
Thus the function $f_1 - f_1(0)$ satisfies the Cauchy functional equation.
Together with the continuity of $f_1$, we get \eqref{eq220520-1} for some $c_1\in \R$.
Then \eqref{eq220520-2} follows directly from \eqref{eq220515-1} and \eqref{eq220520-1} with $c_2:=c_1+\sigma$.
\end{proof}

The following lemma is an analog of Lemma \ref{lem:220515-2}, but it does not follow from the Cauchy function equation. Notice that we do not assume continuity in the following Lemma.
\begin{lem}\label{lem:220515-2a}
Let $\sigma \in \R \setminus \{0\}$.
If $f_1,f_2:[0,\infty) \to \R$ are two functions satisfying
\begin{align}\label{eq220515-1a}
\expfsmall{-\sigma t} (f_1(s) + f_2(r)) = f_1(s+ t) + f_2(r- t)
\end{align}
for every real $r \ge 0 \ge t \ge -s$, then there are constants $c_1,c_2,c_1',c_2' \in \R$ such that
\begin{align}
&f_1(s)=
c_1 \expfsmall{-\sigma s} + c_1', ~\forall s \ge 0,\label{eq220520-1a} \\
&f_2(r)=
c_2 \expfsmall{\sigma r} + c_2', ~\forall r \ge 0, \label{eq220520-2a}
\end{align}
and $c_1'+c_2'=0$.
\end{lem}
\begin{proof}
Reformulating \eqref{eq220515-1a}, we get
\begin{align*}
f_1(s) - \expfsmall{\sigma t} f_1(s+ t) =  - f_2(r) + \expfsmall{\sigma t}f_2(r - t).
\end{align*}
As the proof of Lemma \ref{lem:220515-2}, there is a function $g$ such that
\begin{align}
&g(t) = f_1(s) - \expfsmall{\sigma t} f_1(s+ t) \label{eq220515-2a}
\end{align}
for every $s \ge -t \ge 0$.

Set $s=- t$ in \eqref{eq220515-2a}.
We have $g(t)=f_1(- t) - \expfsmall{\sigma t} f_1(0)$ and thus
\begin{align*}
f_1(s) = \expfsmall{\sigma t} f_1(s+ t) + f_1(-t)  - \expfsmall{\sigma t} f_1(0).
\end{align*}
Denote $a:=s+ t$ and $b:= -t$. We get
\begin{align*}
f_1(a+b) = \expfsmall{-\sigma b} f_1(a) + f_1(b)  - \expfsmall{-\sigma b} f_1(0)
\end{align*}
for every $a,b \ge 0$.
Swapping $a$ and $b$ in the above relation, we get
\begin{align*}
f_1(a+b) = \expfsmall{-\sigma a} f_1(b) + f_1(a)  - \expfsmall{-\sigma a} f_1(0).
\end{align*}
Therefore,
\begin{align*}
\ab{\expfsmall{-\sigma b}-1} (f_1(a)-f_1(0))= \ab{\expfsmall{-\sigma a}-1} (f_1(b)-f_1(0)).
\end{align*}
Since $\sigma \neq 0$, there is a constant $c_1 \in \R$ such that
\begin{align*}
\frac{f_1(a)-f_1(0)}{\expfsmall{-\sigma a}-1}=c_1
\end{align*}
for every $a > 0$, which implies \eqref{eq220520-1a}.
Then \eqref{eq220520-2a} follows directly from \eqref{eq220515-1a} and \eqref{eq220520-1a} with $c_2:=c_1'+f_2(0)$ and $c_1'+c_2'=0$.
\end{proof}

Sometimes we do not care about $\valf \u(x)$ of $x=o$.
Denote by $\FFo{n}$ the set of functions $f:\ro{n} \to \R$.

\begin{lem}\label{lem:220512-1}
If a transform $\valf :\convs \rightarrow \FFo{n}$ is continuous and $\sln$ contravariant, and there is a function $\mathcal{A}:\R^n \to \R^n$ such that
\begin{align}\label{eq220511-1}
\valf (\tau_y \u)=\valf (\u)+\lf_{\mathcal{A}(y)}
\end{align}
or
\begin{align}\label{eq220511-1a}
\valf (\tau_y \u)=\expf{\lf_{\mathcal{A} (y)}} \valf (\u)
\end{align}
for every $\u\in \convs$ and $y \in \R^n$,
then there is a constant $\sigma \in \R$ such that
\begin{align*}
\mathcal{A}(y) = \sigma y.
\end{align*}
\end{lem}
\begin{proof}
For any $y,\tilde y \in \R$, the hypothesis \eqref{eq220511-1} gives
\begin{align*}
\valf (\u)+\lf_{\mathcal{A}(y+\tilde y)} = \valf (\tau_{y+\tilde y} \u)=\valf (\tau_{\tilde y}\u)+\lf_{\mathcal{A}(y)}
= \valf (\u) + \lf_{\mathcal{A}(\tilde y)} +  \lf_{\mathcal{A}(y)};
\end{align*}
and the hypothesis \eqref{eq220511-1a} gives
\begin{align*}
\expf{\lf_{\mathcal{A} (y +\tilde y)}} \valf (\u) = \valf (\tau_{y+\tilde y} \u)=\expf{\lf_{\mathcal{A} (y)}}  \valf (\tau_{\tilde y}\u)
= \expf{\lf_{\mathcal{A} (y)}+\lf_{\mathcal{A}(\tilde y)}} \valf (\u).
\end{align*}
In both settings, we have
\begin{align*}
\mathcal{A}(y+\tilde y) = \mathcal{A}(y)+\mathcal{A}(\tilde y).
\end{align*}
For $\{y_i\}_{i=1}^\infty \subset \R^n$, we have $\tau_{y_i}\u \epic \tau_y \u$ when $y_i \to y$.
Thus the continuity of $\valf$ and \eqref{eq220511-1} (or \eqref{eq220511-1a}) imply that $\mathcal{A}(\cdot)$ is continuous.
By Lemma \ref{lem:220515-1}, $\mathcal{A}$ is a linear transformation on $\R^n$, that is, there is $A \in \R^{n \times n}$ such that
\begin{align}\label{eq220511-2}
\mathcal{A}(y)=Ay.
\end{align}

Since $\valf$ is $\sln$ contravariant, we have
\begin{align}\label{eq220512-2}
\valf (\tau_{\phi^{-1} y} \u)(\phi^t x)
=\valf ((\tau_{\phi^{-1} y} \u) \circ \phi^{-1})(x)
=\valf (\tau_{y} (\u\circ \phi^{-1}))(x)
\end{align}
for every $\phi \in \sln$.
Together with \eqref{eq220511-2} and the hypothesis \eqref{eq220511-1} (or \eqref{eq220511-1a}), we get
\begin{align*}
\valf \u(\phi^t x) + (A \phi^{-1} y) \cdot (\phi^t x)
= \valf (\u\circ \phi^{-1})(x) + (A y) \cdot x,
\end{align*}
and
\begin{align*}
\expf{(A \phi^{-1} y) \cdot (\phi^t x)}\valf \u(\phi^t x)
= \expf{ (A y) \cdot x} \valf (\u\circ \phi^{-1})(x).
\end{align*}
Now the $\sln$ contravariance of $\valf$ and the arbitrariness of $x,y$ give
\begin{align*}
\phi A \phi^{-1} =A,  ~\forall \phi \in \sln,
\end{align*}
whence $A$ is a scalar matrix (see e.g., \cite[p541]{MR1878556}),
which completes the proof.
\end{proof}

A function $\valf_0:\convs \to \R$ is \emph{$\sln$ invariant} if
\begin{align*}
\valf_0(\u \circ \phi^{-1})=\valf_0(\u)
\end{align*}
for every $\u\in \convs$ and $\phi \in \sln$;
it is  \emph{translation invariant} if
\begin{align*}
\valf_0(\tau_y \u)=\valf_0(\u)
\end{align*}
for every $\u\in \convs$ and $y \in \R^n$.

\begin{lem}\label{lem:220329-1}
If a transform $\valf :\convs \rightarrow \FFo{n}$ is continuous and $\sln$ contravariant, and there is a function $\valf_0:\convs \to \R$ such that
\begin{align}\label{eq220512-1}
\valf (\tau_y \u)=\valf (\u)+\lf_{\valf_0(\u) y},
\end{align}
or
\begin{align*}
\valf (\tau_y \u)=\expf{\lf_{\valf_0(\u) y}} \valf (\u),
\end{align*}
for every $\u\in \convs$ and $y \in \R^n$,
then $\valf_0$ is continuous, $\sln$ invariant and translation invariant.
\end{lem}
\begin{proof}
We only prove the first case since the second case is similar.
Let $x \cdot y \neq 0$.
Notice that \eqref{eq220512-2} holds.
By \eqref{eq220512-2} and \eqref{eq220512-1}, we have
\begin{align*}
\valf (\u \circ \phi^{-1} )(x) + \valf_0(\u \circ \phi^{-1}) x \cdot y
&=\valf (\tau_y (\u \circ \phi^{-1}))(x) \\
&=\valf (\tau_{\phi^{-1} y} \u)(\phi^t x) \\
&=\valf (\u)(\phi^t x) + \valf_0(\u) x \cdot y
\end{align*}
for every $\u \in \convs$ and $\phi \in \sln$.
Together with the $\sln$ contravariance of $\valf$, it gives that $\valf_0$ is $\sln$ invariant.
Also, \eqref{eq220512-1} gives
\begin{align*}
\valf (\tau_{y}\tau_{\tilde y} \u)(x) = \valf (\tau_{\tilde y} \u)(x) + \valf_0(\tau_{\tilde y} \u) x \cdot y = \valf (\u)(x) + \valf_0(\u) x \cdot \tilde y + \valf_0(\tau_{\tilde y} \u) x \cdot y
\end{align*}
and
\begin{align*}
\valf (\tau_{y}\tau_{\tilde y} \u)(x)=\valf (\tau_{y+\tilde y} \u)(x) = \valf (\u)(x)+\valf_0(\u) x \cdot (y+\tilde y).
\end{align*}
Thus $\valf_0$ is translation invariant.
Similarly, $\valf_0$ is continuous.
\end{proof}

The proof of the following lemma is similar to Lemma \ref{lem:220329-1}.
Hence we omit the proof.
\begin{lem}\label{lem:220329-1a}
If a transform $\valf :\convs \rightarrow \FFo{n}$ is a valuation and there is a function $\valf_0:\convs \to \R$ such that
\begin{align*}
\valf (\tau_y \u)=\valf (\u)+\lf_{\valf_0(\u) y}
\end{align*}
for every $\u\in \convs$ and $y \in \R^n$,
then $\valf_0$ is a valuation.
\end{lem}

\section{Valuations on convex bodies}\label{sec:vb}
Let $\MQ^n$ be a subset of $\MK^n$.
We call a map $Z:\MQ^n \to \FFo{n}$ is a \emph{valuation} if
\begin{align*}
Z K + Z L = Z (K \cup L) + Z(K \cap L),
\end{align*}
whenever $K,L,K \cup L, K \cap L \in \MQ^n$.
We call $Z$ \emph{continuous} if $ZK_i (x) \to ZK(x)$ for every $x\in \ro{n}$ whenever $K_i,K \in \MQ^n$ and $K_i \to K$ with respect to the Hausdorff metric.

A map $Z:\MQ^n \to \FFo{n}$ is called \emph{$\sln$ covariant} if
\begin{align*}
Z(\phi P)(x)=ZP(\phi^t x)
\end{align*}
for every $P \in \MQ^n$, $x \in \ro{n}$ and $\phi \in \sln$.

Denote by $\MP^n$ the set of all convex polytopes in $\R^n$.
Let $V_0$ be the Euler characteristic, $V_n$ be the $n$-dimensional volume, $m(P) = \int_P x \dif x$ be the moment vector of $P$, and $\lap P (x)= \int_{P} \expf{x \cdot y}\dif y$.

The main results established in this section are the following theorems.
\begin{thm}\label{thm:id1}
Let $n\ge 3$. A map $Z: \MP^n \to \FFo{n}$ is a continuous and $\sln$ covariant valuation such that there is a function $Z_0: \MP^n \to \R$ for which
\begin{align}\label{eq:thm:id1-1}
Z(P+y)(x)=Z(P)(x) +  Z_0(P)  x \cdot y
\end{align}
for every $P \in \MP^n$, $x \in \ro{n}$, and $y \in \R^n$,
if and only if
there are constants $c_1,c_2,c_3,c_4,c_5 \in \R$ such that
\begin{align*}
Z_0(P)=(c_1-c_2)V_0(P)+c_5V_n(P)
\end{align*}
and
\begin{align}\label{eq:id1}
ZP(x)= c_1 h_{P}(x) +c_2 h_{-P}(x) + c_3V_0(P) +c_4 V_n(P)+ c_5 x \cdot m(P)
\end{align}
for every $P \in \MP^n$ and $x \in \ro{n}$.
\end{thm}

\begin{thm}\label{thm:ltn3}
Let $n \ge 3$. A map $Z: \MP^n \to \FFo{n}$ is a continuous and $\sln$ covariant valuation such that there is a function $Z_0: \MP^n \to \R \setminus \{0\}$ for which 
\begin{align}\label{def:logt2a}
Z(P+y)(x)=ZP(x) \expf{Z_0(P) x \cdot y}
\end{align}
for every $P \in \MP^n$, $x \in \ro{n}$, and $y \in \R^n$,
if and only if
there are constants $c_1,c_2,c_3 \in \R$ and $\sigma \in \R \setminus \{0\}$ such that
\begin{align*}
Z_0(P)=\sigma
\end{align*}
and
\begin{align}\label{val:slco}
ZP(x)=c_1\expf{\sigma h_{P}(x)}+c_2\expf{-\sigma h_{-P}(x)} + c_3\lap P (\sigma x)
\end{align}
for every $P \in \MP^n$ and $x \in \ro{n}$.
\end{thm}

Remark: similar to Lemma \ref{lem:220512-1}, we can replace $Z_0(P)  x \cdot y$ in above theorems by $x \cdot \mathcal{A}(y)$ with a function $\mathcal{A}:\R^n \to \R^n$ (Then $c_5=0$ in \eqref{eq:id1}).

The ``if" parts of those theorems are easy. Therefore, we only deal with the ``only if" parts.

Valuations satisfying \eqref{eq:thm:id1-1} are called \emph{translation covariant valuations} and valuations satisfying \eqref{def:logt2a} are called \emph{log-translation covariant valuations}.

The continuity of a valuation $Z$ and the $\sln$ covariance easily imply that functions in the image of $Z$ have some continuity properties; see \cite[Lemma 9]{Li2020slnco}. That is:
\begin{lem}\label{lem:contin}
Let $n \ge 3$. If $Z: \MP^n \to \FFo{n}$ is continuous and $\sln$ covariant, then $ZP \in \CFo{n}$ for every $P \in \MP^n$.
\end{lem}

A map $Z_0: \MP^n \to \R$ is $\sln$ invariant if
\begin{align*}
Z_0(\phi P)=ZP
\end{align*}
for every $P \in \MP^n$ and $\phi \in \sln$;
it is \emph{translation invariant} if
\begin{align*}
Z_0(P+y)=Z_0(P)
\end{align*}
for every $P \in \MP^n$ and $y \in \R^n$.
The following lemma helps us to determine $Z_0$ in Theorem \ref{thm:id1} (but not for $Z_0$ in Theorem \ref{thm:ltn3}).
\begin{lem}[Blaschke; see also Ludwig, Reitzner \cite{Li2020slnco}]\label{lem:MR}
Let $n \ge 2$. A map $Z_0: \MP^n \to \R$ is a continuous, translation invariant and $\sln$ invariant valuation
if and only if
there are $c,c' \in \R$ such that
\begin{align*}
Z_0(P)= cV_0(P)+c' V_n(P)
\end{align*}
for every $P \in \MP^n$.
\end{lem}

The following Lemma is crucial in this section.
A signed Radon measure is called continuous if the measure of every singleton is null.
We denote by $\cms$ the space of signed and continuous Radon measures on $\R$.

\begin{lem}[Li \cite{Li2020slnco}]\label{lem:3body}
Let $n \ge 3$. A map $Z: \MPon \to \FFo{n}$ is a continuous and $\sln$ covariant valuation
if and only if
there are $\zeta \in \CF{}$ and $\mu \in \cms$ such that
\begin{align*}
ZP(x)&= \zeta(h_P(x)) +\zeta(-h_{-P}(x)) + \frac{1}{|x|} \int_{\R} V_{n-1}(P \cap H_{x,t}) \dif \mu(t)
\end{align*}
for every $P \in \MPon$ and $x \in \ro{n}$.
\end{lem}

Let $\eta:\R\to \R$ be a locally integrable function. We remark that
\begin{align}\label{eq:fub}
\frac{1}{|x|} \int_{\R} V_{n-1}(P \cap H_{x,t}) \eta(t) \dif t = \int_{P} \eta(x \cdot y) \dif y,
\end{align}
which follows from Fubini's theorem.

\subsection{Translation covariant valuations}\label{sec:n3t}

\begin{proof}[Proof of Theorem \ref{thm:id1}]
The proof is divided into four steps.

\myvskip
\textbf{Step i.}
Similar to Lemma \ref{lem:220329-1} and Lemma \ref{lem:220329-1a}, it is easy to see that the map $P \mapsto Z_0(P)$ is a continuous, translation invariant and $\sln$ invariant, real-valued valuation on $\MP^n$.
By Lemma \ref{lem:MR}, there are $c,c' \in \R$ such that
\begin{align}\label{eq:thm:id1-2}
Z_0(P)=cV_0(P) + c' V_{n}(P)
\end{align}
for every $P \in \MP^n$.

\myvskip
\textbf{Step ii.}
Setting $o \in P \cap (P+y)$ in Lemma \ref{lem:3body}, we have
\begin{align*}
Z(P+y)(x)= \zeta(h_P(x)+x \cdot y) + \zeta(-h_{-P}(x)+x \cdot y) + \frac{1}{|x|} \int_{\R} V_{n-1}((P+y) \cap H_{x,t}) \dif\mu(t)
\end{align*}
and
\begin{align*}
Z(P)(x)= \zeta(h_P(x)) +\zeta (-h_{-P}(x)) + \frac{1}{|x|} \int_{\R} V_{n-1}(P \cap H_{x,t}) \dif\mu(t).
\end{align*}
Choose $P =[-re_1,se_1]$, $x=e_1$ and $y=de_1$ such that $r \ge 0 \ge d \ge -s$.
Therefore, $o \in P \cap (P+y)$, $h_P(x)=s$, $h_{-P}(x)=r$ and $x \cdot y=d$.
Together with \eqref{eq:thm:id1-1} and \eqref{eq:thm:id1-2}, the above two relations imply
\begin{align*}
\zeta(s+d)+\zeta(-r+d)=\zeta(s)+\zeta(-r)+ c d.
\end{align*}
Set $f_1(s):=\zeta(s),~\forall s \ge 0$ and $f_2(r):=\zeta(-r),~\forall r \ge 0$.
We get
\begin{align*}
f_1(s)+f_2(r)+cd=f_1(s+d)+f_2(r-d)
\end{align*}
for every real $r \ge 0 \ge d \ge -s$.
By Lemma \ref{lem:220515-2}, there are $c_1,c_2 \in \R$ such that
\begin{align*}
\zeta(s)=f_1(s) = c_1 s + \zeta(0)
\end{align*}
for any $s \ge 0$;
\begin{align*}
\zeta(-r)=f_2(r)=c_2 r + \zeta(0)
\end{align*}
for any $r \ge 0$;
and
\begin{align}\label{eq1022-1}
c =c_1-c_2.
\end{align}
Denote $c_3:=2\zeta(0)$.
Together with Lemma \ref{lem:3body}, we have
\begin{align}\label{eq:thm:id1-6}
ZP(x)=c_1 h_P(x) + c_2 h_{-P}(x) + c_3V_0(P) + \frac{1}{|x|} \int_{\R} V_{n-1}(P \cap H_{x,t}) \dif \mu(t)
\end{align}
for any $P \in \MPon$ and $x \in \ro{n}$.

\myvskip
\textbf{Step iii.}
Consider
\begin{align*}
Z'P(x) = \frac{1}{|x|} \int_{\R} V_{n-1}(P \cap H_{x,t}) \dif \mu(t), ~x\in \ro{n}.
\end{align*}
By \eqref{eq:thm:id1-1}, \eqref{eq:thm:id1-2}, \eqref{eq1022-1} and \eqref{eq:thm:id1-6}, we have
\begin{align*}
Z'(P+y)(x)=Z'(P)(x) + c' V_n(P) x \cdot y, ~~x\in \ro{n}
\end{align*}
for every $P \in \MPon$ and $y \in -P$.
For real $r \ge 0 \ge d \ge -s$, choose $P=[-r,s]^n$, $y \in -P$, $x=e_1$ and $d=x \cdot y$.
Recalling that $\mu \in \cms$, we have
\begin{align*}
\mu(-r+d,s+d) = \mu(-r,s) + c' (r+s)d.
\end{align*}
Set $c_5=c'$, define $\tilde \mu \in \cms$ by $\tilde \mu(a,b)=\mu(a,b) - \int_{a}^b c_5 t \dif t$ for any $-\infty< a \le b < \infty$,
and set $f_1(s):=\tilde \mu(0,s)$ for any $s \ge 0$ and $f_2(r):=\tilde \mu(-r,0)$ for any $r \ge 0$.
Since $\tilde \mu\in \cms$, we obtain $f_1(0)=f_2(0)=0$,
\begin{align*}
f_1(s+d)+f_2(r-d)=f_1(s)+f_2(r),
\end{align*}
for every real $r \ge 0 \ge d \ge -s$, and that $f_1,f_2$ are continuous.
By Lemma \ref{lem:220515-2}, there is $c_4 \in \R$ such that
$f_1(s) =c_4s = f_2(s)$ for every $s \ge 0$.
Observe that $\tilde \mu$ is uniquely determined by $f_1$ and $f_2$, thus $\dif \tilde \mu (t) = c_4 \dif t$ for some $c_4 \in \R$.
By \eqref{eq:fub}, it turns out
\begin{align}\label{eq:thm:id1-8}
\frac{1}{|x|} \int_{\R} V_{n-1}(P \cap H_{x,t}) \dif\mu(t)
= c_4 V_n(P) + c_5 \int_{P} x\cdot y \dif y
= c_4 V_n(P) + c_5 x \cdot m(P).
\end{align}

\myvskip
\textbf{Step iv.}
\eqref{eq:thm:id1-6} and \eqref{eq:thm:id1-8} show that \eqref{eq:id1} holds for every $P \in \MPon$.
Now, for any $P\in \MP^n$, choose $y \in P$ and replace $P$ by $P-y$ in \eqref{eq:thm:id1-1}.
Then \eqref{eq:id1} holds for every $P \in \MP^n$ and the proof is completed.
\end{proof}

\subsection{Log-translation covariant valuations}
\begin{proof}[Proof of Theorem \ref{thm:ltn3}]
The proof is divided into five steps.

\myvskip
\textbf{Step i.} Similar to Lemma \ref{lem:220329-1}, it is easy to see that $Z_0$ is translation invariant and $\sln$ invariant.
Therefore,
\begin{align}\label{eq1025-1}
0 \neq Z_0([-r,s])=Z_0([-1,1])=:\sigma
\end{align}
for every $r,s \ge 0$.

\myvskip

\textbf{Step ii.}
Lemma \ref{lem:3body}, \eqref{def:logt2a} and \eqref{eq1025-1} imply that
\begin{align*}
\zeta(h_{P+y}(x)) +\zeta(-h_{-P-y}(x))
&=Z(P+y)(x) \notag \\
&= \expf{\sigma x \cdot y}ZP(x)   \notag \\
&=\expf{\sigma x \cdot y} \left(\zeta(h_P(x)) +\zeta(-h_{-P}(x))\right)
\end{align*}
for $P=[-re_1,se_1]$, $x=e_1$ and $y=de_1$ with $r,s \ge 0$ and $r \ge 0 \ge d \ge -s$.
Now we get
\begin{align*}
\zeta(s+d)+\zeta(-r+d) = \expf{\sigma d}(\zeta(s) + \zeta(-r)).
\end{align*}
Set $f_1(s):=\zeta(s)$ and $f_2(r)= \zeta(-r)$.
We get
\begin{align*}
\expf{\sigma d}(f_1(s) + f_2(r))=f_1(s + d) + f_2(r -d)
\end{align*}
for every real $r \ge 0 \ge d \ge -s$.
By Lemma \ref{lem:220515-2a}, there are $c_1,c_2,c_1',c_2' \in \R$ such that
\begin{align*}
&\zeta(s)=f_1(s)=c_1 \expf{\sigma s} + c_1',\\
&\zeta(-r)=f_2(r)=c_2 \expf{-\sigma r} + c_2',
\end{align*}
for every $s \ge 0$ and $r \ge 0$ and $c_1'+c_2'=0$.
Thus
\begin{align}\label{eq1206-1}
\zeta(h_P(x))+\zeta(-h_{-P}(x)) = c_1 \expf{\sigma h_P(x)} + c_2 \expf{-\sigma h_{-P}(x)}
\end{align}
for every $P \in \MPon$;
and hence by Lemma \ref{lem:3body} and \eqref{def:logt2a},
\begin{align*}
&\ab{c_1 \expf{\sigma h_P(x)} + c_2 \expf{-\sigma h_{-P}(x)}} \expf{\sigma x \cdot y} \\
&= Z(P+y)(x)\\
&= \expf{Z_0(P) x \cdot y}  ZP(x) \\
&= \ab{c_1 \expf{\sigma h_P(x)} + c_2 \expf{-\sigma h_{-P}(x)}} \expf{Z_0(P) x \cdot y}
\end{align*}
for every $P \in \MPon$ with $\dim P <n$.
If $c_1$ or $c_2$ is not zero, we have
\begin{align}\label{eq1211-2}
Z_0(P)=\sigma;
\end{align}
if $c_1=c_2=0$, then $Z(P)(x)=0$, we can still assume that \eqref{eq1211-2} holds.

\myvskip
\textbf{Step iii.} We claim that
\begin{align}\label{eq1211-3}
Z_0([-r,s] \times [-1/2,1/2]^{n-1})=\sigma
\end{align}
for any $r,s \ge 0$.

\myvskip
Indeed, for the case $c_1=c_2=0$, by \eqref{eq:fub} and since $\mu \in \cms$ can be approximated by functions of $L^1(\R)$ in the sense of weak(*) topology of Radon measures, we observe that
\begin{align*}
\frac{1}{|x|} \int_{\R} V_{n-1}((\phi P+y) \cap H_{x,t}) \dif \mu(t)=|\det\phi| \frac{1}{|\phi^t x|} \int_{\R} V_{n-1}((P+\phi^{-1}y) \cap H_{\phi^t x,t}) \dif\mu(t)
\end{align*}
for any $\phi \in \gln$, $P\in \MP^n$, $x\in \ro{n}$ and $y \in \R^n$.
Together with \eqref{def:logt2a}, \eqref{eq1206-1} and Lemma \ref{lem:3body},
we obtain that $Z_0(\phi P)=Z_0(P)$ for any $P \in \MP^n$ and $\phi \in \gln$.
Recalling that $Z_0$ is translation invariant, we have
$Z_0([-r,s] \times [-1/2,1/2]^{n-1})=Z_0 ([-1,1] \times [-1/2,1/2]^{n-1})$ for every $r,s \ge 0$.
Therefore, we trivially say that \eqref{eq1211-3} holds by the last remark in Step ii.

In the following, assume that one of $c_1,c_2$ is not zero.
By the translation invariance of $Z_0$, we denote $\sigma'(r+s)=Z_0([-r,s] \times [-1/2,1/2]^{n-1})$.
To show that \eqref{eq1211-3} holds, we only need
\begin{align*}
\sigma'(s)= \sigma
\end{align*}
for any $s>0$.
By \eqref{def:logt2a}, \eqref{eq1211-2}, and the valuation property, we have
\begin{align*}
&Z([-s,s]\times [-1/2,1/2]^{n-1})(x)\expf{\sigma'(2s) x\cdot y} +Z(\{0\}\times [-1/2,1/2]^{n-1})(x)\expf{\sigma x\cdot y} \\
&=Z([-s,s]\times [-1/2,1/2]^{n-1}+y)(x)+Z(\{0\}\times [-1/2,1/2]^{n-1}+y)(x)\\
&=Z([-s,0]\times [-1/2,1/2]^{n-1}+y)(x)+Z([0,s]\times [-1/2,1/2]^{n-1}+y)(x) \\
&=Z([-s,0]\times [-1/2,1/2]^{n-1})(x)\expf{\sigma'(s) x\cdot y}
  +Z([0,s]\times [-1/2,1/2]^{n-1})(x)\expf{\sigma'(s) x\cdot y}
\end{align*}
for any $s>0$, $x\in\ro{n}$, and $y \in \R$.
Let first $x=re_1$ for $r>0$ and $d=x\cdot y$.
Together with Lemma \ref{lem:3body} and \eqref{eq1206-1}, we have
\begin{align*}
&\ab{c_1e^{\sigma rs}+c_2e^{-\sigma rs}+\frac{1}{r}\mu(-rs,rs)}e^{\sigma'(2s) d} + \ab{c_1+c_2}e^{\sigma d} \notag\\
&=\ab{c_1+c_2e^{-\sigma rs}+\frac{1}{r}\mu(-rs,0)}e^{\sigma'(s) d} + \ab{c_1e^{\sigma rs}+c_2+\frac{1}{r}\mu(0,rs)}e^{\sigma'(s) d}.
\end{align*}
Similarly, let $x=2rs e_n$ for $r>0$ and keep $d=x\cdot y$.
We have
\begin{align*}
&\ab{c_1e^{\sigma rs}+c_2e^{-\sigma rs}+\frac{1}{r}\mu(-rs,rs)}e^{\sigma'(2s) d} + \ab{c_1e^{\sigma rs}+c_2e^{-\sigma rs}}e^{\sigma d} \notag\\
&=\ab{c_1e^{\sigma rs}+c_2e^{-\sigma rs}+\frac{1}{2r}\mu(-rs,rs)}e^{\sigma'(s) d} + \ab{c_1e^{\sigma rs}+c_2e^{-\sigma rs}+\frac{1}{2r}\mu(-rs,rs)}e^{\sigma'(s) d}.
\end{align*}
Recalling $\mu(0)=0$, the difference of above two relations gives
\begin{align*}
\ab{c_1 \ab{1-e^{\sigma rs}} + c_2 \ab{1-e^{-\sigma rs}}} \ab{e^{\sigma d}-e^{\sigma'(s) d}}=0.
\end{align*}
Since $c_1,c_2$ are not all zero, there must be a $r >0$ such that $c_1 \ab{1-e^{\sigma rs}} + c_2 \ab{1-e^{-\sigma rs}} \neq 0$.
Thus $\sigma'(s)=\sigma$.

\myvskip

\textbf{Step iv.}
By \eqref{def:logt2a}, \eqref{eq1206-1}, \eqref{eq1211-3} and Lemma \ref{lem:3body},
\begin{align*}
\frac{1}{|x|} \int_{\R} V_{n-1}((P+y) \cap H_{x,t}) \dif\mu(t) =\expf{\sigma x \cdot y} \frac{1}{|x|} \int_{\R} V_{n-1}(P \cap H_{x,t}) \dif\mu(t)
\end{align*}
for every $P=[-r,s] \times [-1/2,1/2]^{n-1}$ for $r \ge 0,s \ge 0$ and $x=e_1$, $y \in -P$.
Assume $d:=x\cdot y \le 0$.
Thus $r \ge 0 \ge d \ge -s$.
Set $f_1(s)=\mu(0,s)$ for any $s \ge 0$ and $f_2(r)=\mu(-r,0)$ for any $r \ge 0$.
Since $\mu\in \cms$, we get
\begin{align*}
f_1(s+d)+f_2(r-d)=\expf{\sigma d}(f_1(s)+f_2(r)),
\end{align*}
and that $f_1,f_2$ are continuous.
By Lemma \ref{lem:220515-2a}, there are $c_3,c_4,c_3',c_4' \in \R$ such that
\begin{align*}
&f_1(s)=\frac{c_3}{\simga} \expfsmall{\sigma s} + \frac{c_3'}{\simga},\\
&f_2(r)=\frac{c_4}{\simga} \expfsmall{-\sigma r} + \frac{c_4'}{\simga},
\end{align*}
for every $s \ge 0$ and $r \ge 0$ and $c_3'+c_4'=0$.
Observe that $f_1(0)=f_2(0)=0$, we further have
\begin{align*}
c_3=-c_3'=c_4'=-c_4.
\end{align*}
Thus $\dif \mu(t)= c_3 \expfsmall{\sigma t}$.
Together with \eqref{eq:fub}, it turns out
\begin{align}\label{eq0507-1}
\frac{1}{|x|} \int_{\R} V_{n-1}(P \cap H_{x,t}) \dif \mu(t) = c_3\lap P (\sigma x)
\end{align}
with some constant $c_3 \in \R$.

\myvskip
\textbf{Step v.}
By \eqref{eq1206-1}, \eqref{eq0507-1} and Lemma \ref{lem:3body}, we get that \eqref{val:slco} holds for every $P \in \MPon$ and then $Z_0(P)=\sigma$ for every $P \in \MPon$.
Similar to Step iv in the proof of Theorem \ref{thm:id1}, we get \eqref{val:slco} for every $P \in \MP^n$. Then $Z_0(P)=\sigma$ for every $P \in \MP^n$.
\end{proof}

\section{Main results and generalizations}\label{sec:mf}

We require the following lemma of \cite{CLM2019homogeneous}.
\begin{lem}\label{lem:unq}
If $\valf : \convs \to \R$ is a continuous valuation, then it is uniquely determined by $\valf (\indf_P +\lf_y+t)$ for every $P \in \MP^n$, $y \in \R^n$ and $t \in \R$.
\end{lem}

\subsection{Generalizations of Theorem \ref{thm:Leg}}
We prove the following two generalizations of Theorem \ref{thm:Leg}.
\begin{thm}\label{thm:lt2}
Let $n \ge 3$. A transform $\valf :\convs \rightarrow \FF{n}$ is a continuous and $\sln$ contravariant valuation such that there exist $\e \in \R$ and a function $\mathcal{A}:\R^n \to \R^n$ for which
\begin{align*}
\valf (\tau_y \u)=\valf (\u)+\lf_{\mathcal{A}(y)}, ~\valf (\u+\lf_{y})=\tau_{\e y} \valf (\u),
\end{align*}
for every $\u \in \convs$ and $y \in \R^n$,
if and only if there are constants $c,c',\sigma \in \R$ such that $\mathcal{A}y=\sigma y$ for every $y \in \R^n$ and
\begin{align*}
\valf \u (x)= \begin{cases}
\e \sigma \u^\ast(x/\e) + c, & \e \neq 0,\\
c' \delta_x^0 + c ,&  \e = 0,
\end{cases}
\end{align*}
for every $\u \in \convs$ and $x \in \R^n$.
\end{thm}

\begin{thm}\label{thm:lt2a}
Let $n \ge 3$. A transform $\valf :\convs \rightarrow \FF{n}$ is a continuous and $\sln$ contravariant valuation such that there exist $\e \in \R$ and a function $\valf_0:\convs \to [0,\infty)$ for which
\begin{align*}
\valf (\tau_y \u)=\valf (\u)+\lf_{\valf_0(\u) y},~\valf (\u+\lf_{y})=\tau_{\e y} \valf (\u)
\end{align*}
for every $\u \in \convs$ and $y \in \R^n$,
if and only if there are constants $c,c' \in \R$ and $\sigma \ge 0$ such that $\valf_0 \equiv \sigma$ and
\begin{align*}
\valf \u (x)= \begin{cases}
\e \sigma \u^\ast(x/\e) + c, & \e \neq 0, \\
c' \delta_x^0 + c ,&  \e = 0,
\end{cases}
\end{align*}
for every $\u \in \convs$ and $x \in \R^n$.
\end{thm}

The following lemma follows from Mussnig \cite{MR4201535}.
\begin{lem}[Mussnig \cite{MR4201535}]\label{lem:Mus}
Let $n \ge 2$.
If $\valf_0:\convs \to [0,\infty)$ is a continuous, $\sln$ invariant and translation invariant valuation, then there are continuous functions $\eta_0,\eta_1$ such that
\begin{align}\label{eq220329-1}
\valf_0(\u)=\eta_{0}\left(\min _{w \in \mathbb{R}^{n}} u(w)\right)+\int_{\dom \u} \eta_{1}(u(w)) \dif w
\end{align}
for every $\u \in \convs$.
\end{lem}
Mussnig \cite{MR4201535}[Theorem 1.3] gave an ``if and only if" statement with some restrictions on $\eta_0,\eta_1$.
We do not describe those restrictions since they are not needed herein.

By Lemmas \ref{lem:220512-1}-\ref{lem:220329-1a} and Lemma \ref{lem:Mus}, we can prove Theorem \ref{thm:lt2} and Theorem \ref{thm:lt2a} simultaneously by
proving the following theorem.

\begin{thm}\label{thm:lt2b}
Let $n \ge 3$. A transform $\valf :\convs \rightarrow \FF{n}$ is a continuous and $\sln$ contravariant valuation such that there exist $\e \in \R$ and continuous functions $\eta_0,\eta_1:\R \to \R$ for which \eqref{eq220329-1} holds and
\begin{align}
&\valf (\tau_y \u)=\valf (\u)+\lf_{\valf_0(\u) y}, \label{def:tran}\\
&\valf (\u+\lf_{y})=\tau_{\e y} \valf (\u), \label{def:dualtran}
\end{align}
for every $\u \in \convs$ and $y \in \R^n$,
if and only if there are constants $c,c',\sigma \in \R$ such that $\valf_0 \equiv \sigma$ and
\begin{align*}
\valf \u (x)= \begin{cases}
\e \sigma \u^\ast(x/\e) + c, & \e \neq 0, \\
c' \delta_x^0 + c ,&  \e = 0,
\end{cases}
\end{align*}
for every $\u \in \convs$ and $x \in \R^n$.

Moreover, $\valf_0 \ge 0$ if and only if $\sigma \ge 0$.
\end{thm}

Before proving Theorem \ref{thm:lt2b}, we prove a Lemma of restriction which also holds for the planar case.
\begin{lem}\label{lem:220519-1}
Let $n \ge 2$ and $\valf :\convs \rightarrow \FF{n}$ be a continuous valuation satisfying \eqref{def:dualtran}
with $\e \in \R$.
Further assume that there are constants $c_1,c_2,c_3,c_4,c_5,\sigma \in \R$ with $c_1-c_2=\sigma$ such that
\begin{align}\label{eq220519-1}
\valf (\indf_P+t)(x)=c_1 h_{P}(x) +c_2 h_{-P}(x) - \e \sigma t + c_3 + (-c_5 \e t +c_4) V_n(P)+ c_5 x \cdot m(P)
\end{align}
for every $P \in \MP^n$, $x \in \ro{n}$, and $t \in \R$.
We have
\begin{align*}
\begin{cases}
\text{if~} \e>0, \text{~then~} c_1=\sigma, ~c_2 = c_4=c_5=0; \\
\text{if~} \e<0, \text{~then~} c_2=-\sigma, ~ c_1=c_4=c_5=0;\\
\text{if~} \e=0, \text{~then~} c_1=c_2=c_4=c_5=\sigma=0.
\end{cases}
\end{align*}
\end{lem}
\begin{proof}
Let $i,m >0$ be integers, $u_i=\indf_{[i-1,i] \times [0,1]^{n-1}}+\lf_{ie_1} - \frac{i^2-i}{2}$ and $v_m=u_1 \wedge \dots \wedge u_m$.
We have $\u_i \vee \u_{i+1}=\indf_{i\times [0,1]^{n-1}} + \frac{i^2+i}{2}$.
Hence by \eqref{def:dualtran}, \eqref{eq220519-1}, and the valuation property, we get
\begin{align*}
&\valf \v_m (re_1) \notag\\
&= \sum_{i=1}^m \valf \u_i(re_1)- \sum_{i=1}^{m-1} \valf (\u_i \vee u_{i+1})(re_1) \notag \\
&=\sum_{i=1}^m \valf \ab{\indf_{[i-1,i]\times [0,1]^{n-1}} - \frac{i^2-i}{2}}(re_1- i\e e_1)-\sum_{i=1}^{m-1} \valf \left(\indf_{i\times [0,1]^{n-1}}+ \frac{i^2+i}{2}\right)(re_1),
\end{align*}
where
\begin{align*}
&\valf \ab{\indf_{[i-1,i]\times [0,1]^{n-1}} - \frac{i^2-i}{2}}(re_1-i \e e_1) \\
&=c_1h_{[i-1,i]\times [0,1]^{n-1}}(re_1-i \e e_1)+c_2h_{-[i-1,i]\times [0,1]^{n-1}}(re_1-i \e e_1)+ \frac{i^2-i}{2}\e \sigma  +c_3 \\
& \qquad +\ab{c_5\frac{i^2-i}{2}\e + c_4}+\frac{(2i-1)}{2}c_5 r
\end{align*}
and
\begin{align*}
&\valf \left(\indf_{i\times [0,1]^{n-1}}+ \frac{i^2+i}{2}\right)(re_1)
=i \sigma r-\frac{i^2+i}{2}\e \sigma +c_3
\end{align*}
for any $r,r-i\e \neq 0$.

Assume that $r - i \e  < 0$ for all $i \ge 1$.
We further have
\begin{align*}
&\valf v_m (re_1) \notag\\
&=\sum_{i=1}^m \left(i(r-i\e)\sigma  -c_1(r-i \e) + \frac{i^2-i}{2}\e \sigma
        +c_3+ c_4 + c_5\frac{i^2-i}{2}\e + \frac{2i-1}{2}c_5r \right) \notag\\
&   \qquad \qquad - \sum_{i=1}^{m-1} \left(i\sigma r-\frac{i^2+i}{2}\e \sigma +c_3 \right) \notag \\
&=\frac{m(m+1)(m-1)}{6} c_5 \e + \frac{m^2+m}{2}\ab{-\e \sigma + c_1 \e +c_5r} + \frac{m}{2}(-2c_1r+ 2c_4 - c_5r+2\sigma r)+c_3.
\end{align*}
Define $v(x):=u_i(x)$ for $x \in [i-1,i] \times [0,1]^{n-1}$ and $v(x)= \infty$ for other $x \in \R^n$, where $i=1,2,\dots$.
Note that $v:=\lim_{m \to \infty}^{\epi} v_m \in \convs$ from the definition of epi-convergence.
Thus $\lim_{m \to \infty} \valf \v_m(re_1) = \valf \v(re_1) \neq \pm \infty$.
For $\e \ge 0$, together with $c_1-c_2 = \sigma$, it turns out
\begin{align}\label{eq:0601-2a}
c_1=\sigma, ~c_2 = c_4=c_5=0,
\end{align}
by setting arbitrary $r<0$.

Similarly, assume $r- i \e  > 0$ for all $i \ge 1$.
We have
\begin{align*}
&\valf v_m (re_1) \notag\\
&=\sum_{i=1}^m \left(i(r-i \e)\sigma + c_2(r-i \e)+ \frac{i^2-i}{2}\e \sigma
        +c_3 + c_4 + c_5\frac{i^2-i}{2}\e + \frac{2i-1}{2}c_5r\right) \\
&   \qquad \qquad - \sum_{i=1}^{m-1} \left(i\sigma r-\frac{i^2+i}{2}\e \sigma +c_3 \right) \notag \\
&=\frac{m(m+1)(m-1)}{6} c_5 \e + \frac{m^2+m}{2}\ab{-\e \sigma - c_2 \e +c_5r}+ \frac{m}{2}(2c_2r+ 2c_4 - c_5r+2\sigma r) +c_3.
\end{align*}
For $\e \leq 0$, together with $c_1-c_2 = \sigma$, it turns out
\begin{align}\label{eq:0601-3a}
c_2=-\sigma, ~ c_1=c_4=c_5=0,
\end{align}
by setting arbitrary $r>0$.

For $\e=0$, both \eqref{eq:0601-2a} and \eqref{eq:0601-3a} holds.
Together with $c_1-c_2=\sigma$, we have $\sigma=0$.
\end{proof}

\begin{proof}[Proof of Theorem \ref{thm:lt2b}]
The ``if" part is trivial (with $\valf_0 \equiv \sigma$ and that $\e=0$ implies $\sigma=0$).

We prove the ``only if" part in two steps.
\myvskip
\textbf{Step i.}
\emph{Assertion. There are constants $c_1,c_2,c_3,c_4,c_5,\sigma \in \R$ with $c_1-c_2=\sigma$ such that
\begin{align*}
\valf (\indf_P+t)(x)=c_1 h_{P}(x) +c_2 h_{-P}(x) - \e \sigma t + c_3 + (-c_5 \e t +c_4) V_n(P)+ c_5 x \cdot m(P)
\end{align*}
for every $P \in \MP^n$, $x \in \ro{n}$, and $t \in \R$.}
\emph{Moreover, if $\eta_0,\eta_1 \ge 0$, then $c_5,\sigma \ge 0$.}

\myvskip
Indeed, set $u=\indf_P+\lf_y +t$ for some $P \in \MP^n$, $y \in \R^n$ and $t\in \R$.
It is easy to see that
$$\tau_{\tilde{y}} u = \indf_{P+\tilde{y}} + \lf_y - y \cdot \tilde{y} +t$$ for any $\tilde{y} \in \R^n$.
The assumption that $\valf$ satisfies \eqref{def:dualtran} and \eqref{def:tran} now implies
\begin{align}\label{eq:le1}
\valf (\indf_{P+\tilde{y}} - y \cdot \tilde{y} +t)\left(x-\e y\right) &= \valf (\tau_{\tilde{y}} \u)(x) \notag \\
&= \valf (\u)(x) + \valf_0(\indf_P+\lf_y +t) \tilde{y} \cdot x \notag \\
&=\valf (\indf_P +t) (x -\e y) + \valf_0(\indf_P+\lf_y +t) \tilde{y} \cdot x.
\end{align}
In particular, for $y=o$, we have
\begin{align*}
\valf (\indf_{P+\tilde{y}} +t)(x) =\valf (\indf_P +t) (x) + \valf_0(\indf_P +t) \tilde{y} \cdot x.
\end{align*}
For fixed $t \in \R$, it is easy to see that the map $P \mapsto \valf (\indf_P +t)$ satisfies the assumptions in Theorem \ref{thm:id1}.
Therefore, we have functions $\zeta_1,\dots \zeta_5: \R \to \R$ such that
\begin{align}\label{eq220329-2}
\valf (\indf_P+t)(x)=\zeta_1(t) h_{P}(x) +\zeta_2(t) h_{-P}(x) + \zeta_3(t)V_0(P) +\zeta_4(t) V_n(P)+ \zeta_5(t) x \cdot m(P)
\end{align}
and
\begin{align*}
\valf_0(\indf_P + t) = \zeta_1(t) - \zeta_2(t) + \zeta_5(t) V_n(P).
\end{align*}
for every $P \in \MP^n$ and $x \neq o$.
Together with \eqref{eq220329-1}, we also have
\begin{align}\label{eq220329-4}
\valf_0(\u)=(\zeta_1 - \zeta_2)\left(\min _{w \in \mathbb{R}^{n}} u(w)\right)+\int_{\dom \u} \zeta_5(u(w)) \mathrm{d} w
\end{align}
for every $\u \in \convs$.

Combining \eqref{eq:le1} with \eqref{eq220329-2} and \eqref{eq220329-4}, we get
\begin{align}\label{220319-0}
&\zeta_1(- y \cdot \tilde{y} +t) h_{P+\tilde{y}}\left(x-\e y\right) + \zeta_2(- y \cdot \tilde{y} +t) h_{-P-\tilde{y}}\left(x-\e y\right) + \zeta_3(- y \cdot \tilde{y} +t)  \notag \\
&\qquad \qquad  + \zeta_4(- y \cdot \tilde{y} +t) V_n(P) + \zeta_5(- y \cdot \tilde{y} + t) \ab{x-\e y} \cdot \ab{m(P)+V_n(P)\tilde{y}} \notag \\
&=\zeta_1(t) h_{P}(x -\e y) + \zeta_2(t) h_{-P}(x -\e y) + \zeta_3(t) + \zeta_4(t) V_n(P) + \zeta_5(t) \ab{x-\e y} \cdot m(P) \notag\\
&\qquad \qquad +\ab{(\zeta_1 - \zeta_2)\left(\min_{w \in P} w \cdot y + t\right)+\int_{P} \zeta_5(w \cdot y +t ) \dif w} \tilde{y} \cdot x
\end{align}
for any $x-\e y \neq o$.
Let $P=\{w\}$.
We have
\begin{align*}
&\big(\zeta_1(- y \cdot \tilde{y} +t ) - \zeta_2(- y \cdot \tilde{y} +t) \big) (w +\tilde{y}) \cdot \ab{x-\e y} + \zeta_3(- y \cdot \tilde{y} +t)  \notag \\
&\qquad =\big( \zeta_1(t)- \zeta_2(t) \big) w \cdot \ab{x-\e y} + \zeta_3(t) + \big(\zeta_1\left(w \cdot y + t\right) - \zeta_2\left(w \cdot y + t\right)\big) \tilde{y} \cdot x.
\end{align*}
Setting $w=o$ or $w=-\tilde y$, respectively,
and choosing $x \bot \tilde{y} $, we further have
\begin{align*}
&-\big(\zeta_1(- y \cdot \tilde{y} +t) - \zeta_2(- y \cdot \tilde{y} +t)\big) \e \tilde{y} \cdot y + \zeta_3(- y \cdot \tilde{y} +t) = \zeta_3(t),
\end{align*}
and
\begin{align*}
\zeta_3(- y \cdot \tilde{y} +t)
=\big(\zeta_1(t) - \zeta_2(t)\big) \e \tilde{y} \cdot y
+ \zeta_3(t),
\end{align*}
respectively.
Thus
\begin{align}\label{220319-3}
\zeta_1 - \zeta_2 \equiv \sigma,
\end{align}
and
\begin{align}\label{220319-4}
\zeta_3(t)=- \e \sigma t +c_3,~\forall t \in \R
\end{align}
with some constants $c_3,\sigma \in \R$.

For arbitrary $s\ge -r$, we can find $P \in \MP^n$ of arbitrary dimension such that $s=h_{P}(x -\e y)$ and $r=h_{-P}(x -\e y)$.
Thus \eqref{220319-0} together with \eqref{220319-3} and \eqref{220319-4} implies
\begin{align}\label{220319-5}
&\zeta_1(- y \cdot \tilde{y} +t) s + \zeta_2(- y \cdot \tilde{y} +t) r  + \zeta_4(- y \cdot \tilde{y} +t) V_n(P) \notag \\
&\qquad \qquad \qquad \qquad  + \zeta_5(- y \cdot \tilde{y} + t) \ab{x-\e y} \cdot \ab{m(P)+V_n(P)\tilde{y}} \notag \\
&=\zeta_1(t) s + \zeta_2(t) r + \zeta_4(t) V_n(P) + \zeta_5(t) \ab{x-\e y} \cdot m(P) +\ab{\int_{P} \zeta_5(w \cdot y +t ) \dif w} \tilde{y} \cdot x.
\end{align}

Choosing $\dim P<n$ and $s=0$ or $r=0$ in \eqref{220319-5}, respectively, we get
\begin{align*}
\zeta_1 \equiv c_1, ~\zeta_2 \equiv c_2,
\end{align*}
with some constants $c_1,c_2 \in \R$ and then by \eqref{220319-3},
\begin{align*}
c_1-c_2=\simga.
\end{align*}
Further choosing $\dim P=n$, $y \cdot \tilde{y}=0$ but $x \cdot \tilde{y} \neq 0$ in \eqref{220319-5}, we get
\begin{align*}
\zeta_5(t) V_n(P)
= \int_{P} \zeta_5(w \cdot y +t ) \dif w
\end{align*}
for any $P \in \MP^n$.
Thus (by approximating continuous functions with stair functions),
\begin{align*}
\int_{\R^n} \big(\zeta_5(w \cdot y +t ) - \zeta_5(t) \big) g(w) \dif w =0
\end{align*}
for any $g \in C_c(\R^n)$.
Therefore,
\begin{align*}
\zeta_5 \equiv c_5
\end{align*}
with some constant $c_5 \in \R$.
Back to \eqref{220319-5}, we get
\begin{align*}
&\zeta_4(- y \cdot \tilde{y} +t)  - c_5 \e y \cdot \tilde{y} =\zeta_4(t),
\end{align*}
which implies
\begin{align*}
\zeta_4(t)=-c_5 \e t +c_4,~\forall t \in \R,
\end{align*}
with some constant $c_4 \in \R$.

Moreover, if $\eta_0,\eta_1 \ge 0$, then $c_5, \sigma \ge 0$ follows from \eqref{eq220329-4}.

\myvskip
\textbf{Step ii.} Denote $c:=c_3$. By previous step and Lemma \ref{lem:220519-1}, we have
\begin{align}\label{eq220519-2}
\valf (\indf_P+t)(x)=
\begin{cases}
\sigma h_{P}(x) - \e \sigma t + c, &\e>0, \\
-\sigma h_{-P}(x) -\e \sigma t + c, &\e<0, \\
c, &\e=0,
\end{cases}
\end{align}
for every $P \in \MP^n$, $x \in \ro{n}$, and $t \in \R$.

For $\e \neq 0$, \eqref{def:dualtran} implies $\valf \u (y) = \tau_{(x-y)}\valf \u (x)=\valf (\u +\lf_{(x-y)/\e})(x)$.
Therefore, $\valf \u \in \CF{n}$ follows from the continuity of $\valf$.
Together with \eqref{def:dualtran} and \eqref{eq220519-2}, we get
\begin{align*}
\valf (\indf_P+\lf_y+t)(x) =
    \begin{cases}
    \sigma h_P(x-\e y) -\e \sigma t + c , &\e>0, \\
    -\sigma h_{-P}(x- \e y) -\e \sigma t +c, & \e<0,
    \end{cases}
\end{align*}
for every $P \in \MP^n$ and $x \in \R^n$.
Observe that
\begin{align*}
\e \sigma (\indf_P+\lf_y+t)^\ast (x/\e) = \e \sigma (h_P(x/\e-y)-t)=\sgn(\e) \sigma h_{\sgn(\e) P}(x-\e y) - \e \sigma t.
\end{align*}
By Lemma \ref{lem:unq}, we get
\begin{align*}
\valf \u(x) = \e \sigma u^\ast (x/\e) + c
\end{align*}
for every $\u \in \convs$ and $x \in \R^n$.

For $\e =0$, \eqref{def:dualtran} and \eqref{eq220519-2} show that
\begin{align*}
\valf (\indf_P+\lf_y+t)(x) = c
\end{align*}
for every $P \in \MP^n$ and $x \in \ro{n}$.
By Lemma \ref{lem:unq}, we get $\valf (\u)(x)=c$ for every $x \in \ro{n}$.
Setting $\valf'  (\u)(x):=\valf (\u)(o)$ for any $x \in \ro{n}$ and $\u \in \convs$,
it is easy to see that $\valf'$ is a continuous and $\sln$ contravariant valuation satisfying
\eqref{def:tran} and \eqref{def:dualtran} with $\e=\valf_0=0$.
Thus, the previous result shows $\valf' (\u)(x)\equiv c'+c$ for some constant $c' \in \R$, which completes the proof.
\end{proof}

\subsection{Generalizations of Theorem \ref{mthm:log1} and \ref{mthm:log}}
We prove the following generalizations of both Theorem \ref{mthm:log1} and \ref{mthm:log} by considering their equivalent versions on convex functions.
In fact, let $\valf':\lcsc \to \FF{n}$ be a valuation that satisfies the assumptions in Theorem \ref{mthm:log1} and \ref{mthm:log}.
Define $\valf :\convs \rightarrow \FF{n}$ by $\valf \u = \valf' (\expfsmall{-\u})$ for $\u \in \convs$.
It is easy to see that $\valf$ satisfies the assumptions in the following Theorem \ref{thm:log} with $\e=-\sigma=1$ or $\e=\sigma=1$, respectively.
From $\valf$ to $\valf'$ is similar.

\begin{thm}\label{thm:log}
Let $n \ge 3$. A transform $\valf :\convs \rightarrow \FF{n}$ is a continuous and $\sln$ contravariant valuation such that there exist $\e \in \R$ and a function $\mathcal{A}:\R^n \to \R^n$ with $\mathcal{A} \not\equiv o$ for which
\begin{align}
&\valf (\tau_y \u)=\expf{\lf_{\mathcal{A} (y)}} \valf (\u), \label{def:llt2}\\
&\valf (\u+\lf_{y})=\tau_{\e y} \valf (\u), \label{def:dualtran-log}
\end{align}
for every $\u \in \convs$ and $y \in \R^n$,
if and only if there are constants $c_1,c_2 \in \R$ and $\sigma \neq 0$ such that $\mathcal{A}y=\sigma y$ for every $y \in \R^n$ and
\begin{align*}
\valf (\u)(x) =
\begin{cases}
c_1 \expf{\e \sigma \u^\ast(x/\e)} + c_2\int_{\R^n} \expf{\langle \sigma x,y \rangle -  \e \sigma \u(y)}\dif y, & \e \sigma > 0, \\
c_1 \expf{\e \sigma \u^\ast(x/\e)}, & \e \sigma <0,\\
0,&  \e \sigma = 0 ,
\end{cases}
\end{align*}
for every $\u \in \convs$ and $x \in \R^n$.
\end{thm}

Before proving Theorem \ref{thm:log}, we prove two lemmas of restriction that also hold for the planar case.
\begin{lem}\label{lem:220519-1a}
Let $n \ge 2$ and $\valf :\convs \rightarrow \FF{n}$ be a continuous valuation satisfying \eqref{def:dualtran-log}
with $\e \in \R$.
Further assume that there are constants $c_1,\tilde c_1,c_2,\sigma \in \R$ with $\sigma \neq 0$ such that
\begin{align}\label{eq220519-1a}
\valf (\indf_P+t)(x)=\expf{-\e \sigma t}\big(c_1\expf{\sigma h_{P}(x)}+\tilde c_1\expf{-\sigma h_{-P}(x)}\big)
\end{align}
for every $P \in \MP^n$, $x \in \ro{n}$, and $t \in \R$.
We have
\begin{align*}
\begin{cases}
\text{if~} \e>0, \text{~then~} \tilde c_1=0; \\
\text{if~} \e<0, \text{~then~} c_1=0; \\
\text{if~} \e=0, \text{~then~} c_1=\tilde c_1=0.
\end{cases}
\end{align*}
\end{lem}
\begin{proof}
Case 1). $\e>0$.

Fix $s_1,s_2,t_1,t_2 \in \R$ such that $s_1 < s_2-(t_2-t_1)$ and $t_1 <t_2$.
Set
$\u_s = \indf_{[s_1,s]e_1} + t_1 \wedge \ab{\indf_{[s,s_2]e_1}+\lf_{\frac{t_2-t_1}{s_2-s}e_1}+t_2-\frac{t_2-t_1}{s_2-s}s_2}$ for $s_2-(t_2-t_1) < s < s_2$.

By \eqref{def:dualtran-log}, \eqref{eq220519-1a} and the valuation property, we get
\begin{align*}
\valf \u_s(\e e_1) &=\valf \ab{\indf_{[s_1,s]e_1} + t_1}(\e e_1)
+\valf \ab{\indf_{[s,s_2]e_1}+\lf_{\frac{t_2-t_1}{s_2-s}}+t_2-\frac{t_2-t_1}{s_2-s}s_2}(\e e_1) \\
&\qquad \qquad -\valf \ab{\indf_{se_1 + t_1}}(\e e_1),
\end{align*}
where,
\begin{align*}
&\valf \ab{\indf_{[s_1,s]e_1} + t_1}(\e e_1)= c_1\expf{\e \sigma s - \e \simga t_1} + \tilde c_1\expf{\e \sigma s_1 - \e \simga t_1},\\
&~~ \valf \ab{\indf_{se_1} + t_1}(\e e_1)=
c_1\expf{\e \sigma s - \e \simga t_1} + \tilde c_1\expf{\e \sigma s - \e \simga t_1},
\end{align*}
and
\begin{align*}
&\valf \ab{\indf_{[s,s_2]e_1}+\lf_{\frac{t_2-t_1}{s_2-s}e_1}+t_2-\frac{t_2-t_1}{s_2-s}s_2}(\e e_1) \\
&\qquad \qquad =\valf \ab{\indf_{[s,s_2]e_1}+t_2-\frac{t_2-t_1}{s_2-s}s_2}\ab{\e (1- \frac{t_2-t_1}{s_2-s}) e_1} \\
&\qquad \qquad =c_1\expf{\e \sigma s - \e \simga t_1} + \tilde c_1\expf{\e \sigma s_2 - \e \simga t_2}.
\end{align*}
Thus,
\begin{align*}
\valf \u_s(\e e_1) &= c_1\expf{\e \sigma s - \e \simga t_1}+ \tilde c_1\expf{\e \sigma s_1 - \e \simga t_1} \\
&\qquad  + \tilde c_1\expf{\e \sigma s_2 - \e \simga t_2} - \tilde c_1\expf{\e \sigma s - \e \simga t_1}.
\end{align*}
Notice that
\begin{align*}
\valf \ab{\indf_{[s_1,s_2]e_1} + t_1}(\e e_1)= c_1\expf{\e \sigma s_2 - \e \simga t_1} + \tilde c_1\expf{\e \sigma s_1 - \e \simga t_1}.
\end{align*}
Therefore, $\lim_{s \to s_2}^{\epi}\u_s = \indf_{[s_1,s_2]e_1} + t_1$ and the continuity of $\valf$ imply (since $\e\sigma \neq 0$)
\begin{align*}
\tilde c_1=0.
\end{align*}

\myvskip
Case 2). $\e <0$.

Similar to case 1), we get $c_1=0$.

\myvskip
Case 3). $\e =0$.

\eqref{def:dualtran-log} and \eqref{eq220519-1a} give
\begin{align*}
\valf \ab{\indf_{[0,s]} + \lf_{re_1}}(e_1)=c_1\expf{\sigma s} + \tilde c_1
\end{align*}
for $s \ge 0$.
Also, $\lim_{r \to \infty}^{\epi}\indf_{[0,s]} + \lf_{re_1}= \indf_{o}$.
Thus the continuity of $\valf$ implies $c_1=0$.
Similarly, $\lim_{r \to \infty} \valf \ab{\indf_{[0,s]} + \lf_{re_1}}(-e_1) = \valf (\indf_{o})(-e_1)$ implies that
$\tilde c_1=0$.
\end{proof}

\begin{lem}\label{lem:220616-1}
Let $n \ge 2$ and $\valf :\convs \rightarrow \FF{n}$ be a continuous valuation satisfying \eqref{def:dualtran-log}
with $\e \in \R$.
Further assume that there are constants $c_2,\sigma \in \R$ with $\sigma \neq 0$ such that
\begin{align}\label{eq220616-1}
\valf (\indf_P+t)(x)=c_2 \expf{-\e \sigma t}\lap P(\sigma x)
\end{align}
for every $P \in \MP^n$, $x \in \ro{n}$, and $t \in \R$.
If $\e \sigma \le 0$, then $c_2=0$.
\end{lem}
\begin{proof}
We verify the desired result indirectly.
First, assume $c_2 >0$.
Let $i,m > 0$ be integers, $u_i=\indf_{[i-1,i] \times [0,1]^{n-1}}+\lf_{ie_1} - \frac{i^2-i}{2}$ and $v_m=u_1 \wedge \dots \wedge u_m$.
For $r \in \R$, the properties \eqref{def:dualtran-log} and \eqref{eq220616-1} show
\begin{align*}
\valf \ab{\indf_{[i-1,i]\times [0,1]^{n-1}} +\lf_{ie_1} - \frac{i^2-i}{2}}(re_1)
&= \valf \ab{\indf_{[i-1,i]\times [0,1]^{n-1}} }((r-\e i)e_1) \\
&=c_2 \expf{\frac{i^2-i}{2}\e \sigma} \int_{i-1}^i \expf{\sigma (r-\e i) s } \dif s \\
&\ge c_2 \expf{\frac{i^2-i}{2}\e \sigma} \expf{\sigma (r-\e i) (i-1)} \\
&\ge c_2 \expf{\sigma r}.
\end{align*}
Together with the valuation property, we further get
\begin{align*}
&\valf \v_m (re_1) \notag\\
&= \sum_{i=1}^m \valf \u_i(re_1)- \sum_{i=1}^{m-1} \valf (u_i \vee u_{i+1})(re_1) \notag \\
&=\sum_{i=1}^m \valf \ab{\indf_{[i-1,i]\times [0,1]^{n-1}} +\lf_{ie_1} - \frac{i^2-i}{2}}(re_1)-\sum_{i=1}^{m-1} \valf \left(\indf_{i\times [0,1]^{n-1}}+ \frac{i^2+i}{2}\right)(re_1)\\
&\ge m c_2 \expf{\sigma r},
\end{align*}
where we have used $\valf \left(\indf_{i\times [0,1]^{n-1}}+ t\right)(re_1)=c_2 \expf{-\e \sigma t}\lap (i\times [0,1]^{n-1})(\sigma x)=0$ for any $t \in \R$.
It contradicts $\lim_{m \to \infty}\valf v_m (re_1)=\valf \ab{\lim_{m \to \infty}^{\epi} v_m} (re_1) \neq \pm \infty$.

For $c_2<0$, similar to the above argument, we have $\valf \v_m (re_1) \le m c_2 \expf{\sigma r}$ which also contradicts $\lim_{m \to \infty}\valf v_m (re_1)=\valf \ab{\lim_{m \to \infty}^{\epi} v_m} (re_1) \neq \pm \infty$.
\end{proof}

\begin{proof}[Proof of Theorem \ref{thm:log}]
For the ``if" part, we only prove that
\begin{align}\label{eq220616}
\int_{\R^n} \expf{\langle \sigma x,y \rangle -  \e \sigma \u(y)}\dif y < \infty,
\end{align}
when $\e \sigma >0$.
The other parts follow easily from \S \ref{s2}.
Indeed,
\begin{align*}
\int_{\R^n} \expf{\langle \sigma x,y \rangle -  \e \sigma \u(y)}\dif y = \lap (\expf{-\e \sigma \u})(\sigma x).
\end{align*}
Since $-\e \sigma \u \in \convs$ for fixed $\e \sigma >0$, we conclude \eqref{eq220616} from $\lap (\expf{-\u})(x)<\infty$ for every fixed $\u \in \convs$ and $x \in \R^n$ (which is proved in \S \ref{s2}).

\myvskip
The ``only if" part is analogous to the proof of Theorem \ref{thm:lt2b}.

\textbf{Step i}. \textit{Assertion. There are constants $c_1,\tilde c_1,c_2,\sigma \in \R$ with $\sigma \neq 0$ such that
\begin{align*}
\valf (\indf_P+t) (x)=\expf{-\e \sigma t}\big(c_1\expf{\sigma h_{P}(x)}+\tilde c_1\expf{-\sigma h_{-P}(x)}+c_2\lap P(\sigma x)\big)
\end{align*}
for every $P \in \MP^n$, $x \in \ro{n}$, and $t \in \R$.}

\myvskip
Indeed, similar to the first step in the proof of Theorem \ref{thm:lt2b}, by \eqref{def:dualtran-log}, \eqref{def:llt2} and Lemma \ref{lem:220512-1} (clearly $\sigma \neq 0$), for $u=\indf_P+\lf_y +t$, we have
\begin{align*}
\valf (\indf_{P+\tilde{y}} - y \cdot \tilde{y} +t)\left(x-\e y\right) &= \valf (\tau_{\tilde{y}} \u)(x) \notag \\
&= \expf{\sigma \tilde{y} \cdot x} \valf \u(x)  \notag \\
&=\expf{\sigma \tilde{y} \cdot x} \valf (\indf_P +t) (x -\e y).
\end{align*}
Together with Theorem \ref{thm:ltn3}, there are functions $\zeta_1,\tilde \zeta_1,\zeta_2:\R \to\R$ such that
\begin{align*}
\valf (\indf_P +t) (x)=\zeta_1(t)\expf{\sigma  h_{P}(x)}+\tilde \zeta_1(t)\expf{-\sigma h_{-P}(x)} + \zeta_2(t)\lap K(\sigma x)
\end{align*}
for any $x \neq o$, and thus
\begin{align}\label{eq220320-1}
&\zeta_1(- y \cdot \tilde{y} +t)\expf{\sigma  h_{(P+\tilde{y})}(x-\e y)}+\tilde \zeta_1(- y \cdot \tilde{y} +t)\expf{-\sigma h_{-(P+\tilde{y})}(x-\e y)} \notag\\
& \qquad \qquad + \zeta_2(- y \cdot \tilde{y} +t)\int_{P+\tilde{y}} \expf{\sigma (x-\e y) \cdot w} \dif w  \notag\\
&=\expf{\sigma \tilde{y} \cdot x} \ab{\zeta_1(t)\expf{\sigma h_{P}(x-\e y)}+\tilde \zeta_1(t)\expf{-\sigma h_{-P}(x-\e y)}} \notag\\
& \qquad \qquad  + \expf{\sigma \tilde{y} \cdot x} \ab{\zeta_2(t)\int_{P} \expf{\sigma (x-\e y) \cdot w} \dif w}
\end{align}
for any $x - \e y \neq o$.

For arbitrary $s\ge -r$, we can find an arbitrary dimensional $P \in \MP^n$ such that $s=h_{P}(x -\e y)$ and $r=h_{-P}(x -\e y)$.
Therefore, \eqref{eq220320-1} implies
\begin{align*}
&\ab{\zeta_1(- y \cdot \tilde{y} +t)\expf{\sigma s}+\tilde \zeta_1(- y \cdot \tilde{y} +t)\expf{-\sigma r}}  \expf{- \e \sigma \tilde{y}\cdot y}  \\
& \qquad \qquad + \ab{\zeta_2(- y \cdot \tilde{y} +t)\int_{P} \expf{\sigma (x-\e y) \cdot w} \dif w} \expf{- \e \sigma \tilde{y}\cdot y}  \notag\\
&=\ab{\zeta_1(t)\expf{ \sigma s}+\tilde \zeta_1(t)\expf{-\sigma r}+\zeta_2(t)\int_{P} \expf{\sigma (x-\e y) \cdot w} \dif w}.
\end{align*}
Letting $\dim P <n$ and $r \to \infty$, we get
\begin{align*}
\zeta_1(t)=c_1 \expf{-\e \sigma t},~\forall t \in \R,
\end{align*}
with some constant $c_1 \in \R$.
It then implies
\begin{align*}
\tilde \zeta_1(t)=\tilde c_1 \expf{-\e \sigma t},~\forall t \in \R,
\end{align*}
with some constant $\tilde c_1 \in \R$.
Moreover, for $\dim P=n$, together with the above two relations, we finally get
\begin{align*}
\zeta_2(t)=c_2 \expf{-\e \sigma t},~\forall t \in \R,
\end{align*}
with some constant $c_2 \in \R$.
This step is completed.

\myvskip
\textbf{Step ii.}
Define $\valf_1:\convs \to \FF{n}$ by
\begin{align*}
\valf_1 \u(x) =
\begin{cases}
c_1 \expf{\e \sigma \u^\ast(x/\e)}, & \e > 0, \\
\tilde c_1 \expf{\e \sigma \u^\ast(x/\e)}, &\e<0,\\
0, & \e =0.
\end{cases}
\end{align*}
for every $\u \in \convs$ and $x \in \R^n$.

By previous step and Lemma \ref{lem:220519-1a},
\begin{align}\label{eq220519-2a}
\valf (\indf_P+t)(x)=
\begin{cases}
c_1\expf{-\e \sigma t}\expf{\sigma h_{P}(x)}, &\e>0, \\
\tilde c_1\expf{-\e \sigma t}\expf{-\sigma h_{-P}(x)}, &\e<0, \\
0, &\e=0,
\end{cases}
\end{align}
for every $x \in \ro{n}$, $t \in \R$ and $P \in \MP(\R^{n-1})$.
Similar to the part after \eqref{eq220519-2} in the proof of Theorem \ref{thm:lt2b}, we get
\begin{align*}
\valf(\u)=\valf_1(\u)
\end{align*}
for every $\u \in \convs$ with $\dom \u  \subset \R^{n-1}$.
Consider $\valf_2:=\valf-\valf_1$.
It is easy to see that $\valf_2$ also satisfies all the assumptions and $\valf_2(\u)=0$ for every $\u \in \convs$ with $\dom \u  \subset \R^{n-1}$.
Thus previous step and Lemma \ref{lem:220616-1} shows
\begin{align*}
\valf_2 (\indf_P+t)(x)=
\begin{cases}
c_2\expf{-\e \sigma t}\lap P(\sigma x), &\e \sigma>0, \\
0, &\e \sigma \le 0, \\
\end{cases}
\end{align*}
for every $x \in \ro{n}$, $t \in \R$ and $P \in \MP^n$.
Thus we also obtain
\begin{align*}
\valf_2(\u)(x)
=\begin{cases}
c_2\int_{\R^n} \expf{\langle \sigma x,y \rangle -  \e \sigma \u(y)}\dif y, &\e \sigma>0, \\
0, &\e \sigma \le 0,
\end{cases}
\end{align*}
which completes the proof.
\end{proof}

\section{Dual valuations and characterizations of the identity transform}\label{sec:dualv}
In this final section, we show how to use dual valuations to get some classifications of valuations on $\convf$ by the main results introduced in the first section.
It is proved in \cite{CLM2017hessian} that the transform $\valf :\convf \rightarrow \FF{n}$ is a continuous valuation
if and only if the transform $\valf^\ast:\convs \rightarrow \FF{n}$, defined by $$\valf^\ast(\u)=\valf (\u^\ast), \forall \u \in \convs,$$
is a continuous valuation.
We call $\valf^\ast$ the \emph{dual valuation} of $\valf$ and vice versa.

The following properties of $\valf$ are dual to the previous properties of $\valf^\ast$ on $\convs$.
We say a transform $\valf :\convf \rightarrow \FF{n}$ is \emph{$\sln$ covariant} if
\begin{align*}
\valf (\u \circ \phi^{-1})= (\valf \u) \circ \phi^{-1}
\end{align*}
for every $u \in \convf$ and $\phi \in \sln$;
it is a \emph{translation homomorphism} if
\begin{align*}
\valf (\tau_y \u)=\tau_y \valf (\u), ~\valf (\u+\lf_{y})=\valf (\u) + \lf_{y}
\end{align*}
for every $u \in \convf$ and $y \in \R^n$;
and it is \emph{continuous} if $\u_i \epic \u$ implies $\valf \u_i \pc \valf \u$.

The following result follows directly from Theorem \ref{thm:Leg}.

\begin{thm}\label{thm:idt}
Let $n \ge 3$. A transform $\valf :\convf \rightarrow \FF{n}$ is a continuous and $\sln$ covariant valuation which is a translation homomorphism, if and only if there is a constant $c \in \R$ such that
\[\valf \u = \u + c\]
for every $\u\in \convf$.
\end{thm}
\begin{proof}[Proof of Theorem \ref{thm:idt}]
Suppose $\valf :\convf \rightarrow \FF{n}$ is a continuous and $\sln$ covariant valuation that is a translation homomorphism.
Then the dual valuation $\valf^\ast:\convs \rightarrow \FF{n}$ is a continuous and $\sln$ contravariant valuation that is a translation conjugation.
By Theorem \ref{thm:Leg},
\begin{align*}
\valf(\u ^\ast) = \valf^\ast(\u) = \legt{\u} + c,~ \forall \u \in \convs
\end{align*}
for some constant $c$.
Thus $\valf \u = \u + c$ for every $\u \in \convf$, which completes the ``only if" part.

The ``if" part is trivial.
\end{proof}

We remark that Hofst\"{a}tter and Knoerr \cite{MR4567496} established a characterization of the identity transform by $\gln$ covariant endomorphisms.

\myvskip
We can deal with the log-concave functions similarly.
Notice that $\u \in \convf$ if and only if the log-concave function $e^{-\u(x)}>0$ for every $x \in \R^n$.
Denote $\lcpos:=\{e^{-\u}: \u \in \convf\}$.
Following directly from the dual valuations on convex functions, the transform $\valf:\lcpos \to \FF{n}$ is a continuous valuation if and only if the \emph{dual valuation} $\valf^\circ:\lcsc \to \FF{n}$ is a continuous valuation.

We say a transform $\valf:\lcpos \to \FF{n}$ is \emph{$\sln$ covariant} if
\begin{align*}
\valf (f \circ \phi^{-1})= (\valf f) \circ \phi^{-1}
\end{align*}
for every $f \in \lcpos$ and $\phi \in \sln$;
it is a \emph{translation homomorphism on log-concave functions} if
\begin{align*}
\valf (\tau_y f)=\tau_y \valf (f),~\valf (\expfsmall{-\lf_{y}} f)=\expfsmall{-\lf_{y}} \valf f,
\end{align*}
for every $f \in \lcpos$ and $y \in \R^n$;
and it is \emph{continuous} if $f_i \hc f$ implies $\valf f_i \pc \valf f$.

Similar to the case of convex functions, the following Theorems follow directly from Theorem \ref{mthm:log1} and \ref{mthm:log}.
\begin{thm}
Let $n \ge 3$. A transform $\valf :\lcpos \rightarrow \FF{n}$ is a continuous and $\sln$ covariant valuation which is a translation homomorphism on log-concave functions if and only if there is a constant $c \in \R$ such that
\begin{align*}
\valf f= c f
\end{align*}
for every $f \in \lcpos$.
\end{thm}

\begin{thm}
Let $n \ge 3$. A transform $\valf :\lcpos \rightarrow \FF{n}$ is a continuous and $\sln$ covariant valuation which satisfies
\begin{align*}
\valf (\tau_y f)=\tau_y \valf (f),~\valf (\expfsmall{-\lf_{y}} f)=\expfsmall{\lf_{y}} \valf f,
\end{align*}
for every $f \in \lcpos$ and $y \in \R^n$, if and only if there are constants $c_1,c_2 \in \R$ such that
\[\valf f= \frac{c_1}{f} + c_2 \lap f^\circ\]
for every $f \in \lcpos$.
\end{thm}

\section*{Acknowledgement}
\addcontentsline{toc}{section}{Acknowledgement}
The author wish to thank the referees for their many valuable remarks that helped to improve the manuscript.
The work was supported in part by the National Natural Science Foundation of China (12201388), and the Austrian Science Fund (FWF) (M2642 and I3027).

\section*{Statements and Declarations}
On behalf of all authors, the corresponding author states that there is no conflict of interest.
This manuscript has no associated data.

\end{document}